\documentclass[11pt]{amsart}

\usepackage{etex}
\usepackage[utf8]{inputenc}
\usepackage{amsmath, amssymb, amsthm, amsfonts}

\usepackage
[
colorlinks=true,
linkcolor=blue,
anchorcolor=blue,
citecolor=blue,
urlcolor=blue,
plainpages=false,
pdfpagelabels
]{hyperref}

\usepackage{bm}

\usepackage[colorlinks=true,linkcolor=blue,anchorcolor=blue,citecolor=blue,urlcolor=blue,plainpages=false,pdfpagelabels]{hyperref}
\usepackage{graphicx}
\usepackage{mathrsfs}
\usepackage{enumerate}
\usepackage[capitalize]{cleveref}
\usepackage[textsize=scriptsize]{todonotes}
\usepackage{subcaption}

\usepackage{xcolor}
\usepackage[all]{xy}
\usepackage{mdwlist}
\usepackage[numbers,sort,compress]{natbib}

\usepackage{microtype}

\usepackage{xspace}

\usepackage{tikz,tikz-cd}
\usetikzlibrary{matrix,arrows,positioning,calc,decorations.markings}
\tikzset{strike thru arrow/.style={
    decoration={markings, mark=at position 0.5 with {
        \draw [-] 
            ++ (-1pt,-2pt) 
            -- ( 1pt, 2pt);}
    },
    postaction={decorate},
}}

\numberwithin{equation}{section}

\newtheorem*{rep@theorem}{\rep@title}
\newcommand{\newreptheorem}[2]{%
\newenvironment{rep#1}[1]{%
 \def\rep@title{#2 \ref{##1}}%
 \begin{rep@theorem}}%
 {\end{rep@theorem}}}
\makeatother

\newtheorem{theorem}[equation]{Theorem}
\newreptheorem{theorem}{Theorem}

\newtheorem{proposition}[equation]{Proposition}
\newtheorem{lemma}[equation]{Lemma} 
\newtheorem{corollary}[equation]{Corollary}
\newreptheorem{corollary}{Corollary}

\theoremstyle{definition}
\newtheorem{example}[equation]{Example}
\newtheorem{remark}[equation]{Remark}

\newtheorem{definition}[equation]{Definition}
\newtheorem{notation}[equation]{Notation}

\newcommand\category[1]{\ensuremath{\mathbf{#1}}}
\DeclareMathOperator{\SRips}{\mathcal{SR}^u}
\DeclareMathOperator{\nSRips}{\mathcal{SR}}

\DeclareMathOperator{\DRips}{\mathcal{DR}^u}
\DeclareMathOperator{\nDRips}{\mathcal{DR}}
\DeclareMathOperator{\ClDR}{\overline{\mathcal{DR}}}

\DeclareMathOperator{\SCech}{{\mathcal{S\check C}}{}^u}
\DeclareMathOperator{\nSCech}{\mathcal{S\check C}}

\DeclareMathOperator{\DCech}{{\mathcal{D\check C}}{}^u}
\DeclareMathOperator{\nDCech}{\mathcal{D\check C}}

\DeclareMathOperator{\Offset}{\mathcal{O}}
\DeclareMathOperator{\DCov}{\mathcal{DO}}

\newcommand{\ClR}{\bar{\mathcal R}}
\newcommand{\Rips}{\mathcal{R}}
\newcommand{\Cech}{\mathcal{\check C}}

\newcommand{\AmbientSpace}{Z}
\newcommand{\Ball}{B}

\DeclareMathOperator{\rank}{rank}

\DeclareMathOperator{\colim}{colim}

\DeclareMathOperator{\hocolim}{hocolim}

\DeclareMathOperator{\ob}{ob}

\DeclareMathOperator{\Sd}{\mathcal{S}}
\DeclareMathOperator{\Deg}{\mathcal{D}}
\DeclareMathOperator{\Bary}{Bary}
\DeclareMathOperator{\Mult}{\mathcal{M}^u}
\DeclareMathOperator{\nMult}{\mathcal{M}}
\DeclareMathOperator{\rect}{Rect}

\DeclareMathOperator{\VR}{\Rips}
\newcommand{\MB}{\mathcal{B}}
\DeclareMathOperator{\Supp}{Supp}

\newcommand{\Pro}{d_{Pr}}
\newcommand{\GPr}{d_{GPr}}
\newcommand{\GW}{d_{GW}}
\newcommand{\Wa}{d_{W}}

\DeclareMathOperator{\Simp}{\mathbf{Simp}}

\newcommand{\Vect}[0]{\category{Vec}}

\newcommand{\Top}[0]{\category{Top}}

\newcommand{\op}[0]{\mathrm{op}}
\newcommand{\R}[0]{\mathbb R}

\newcommand{\B}[1]{\mathcal B\ifthenelse{\equal{#1}{}}{}{_{#1}}}
\newcommand{\Bopen}[1]{U\ifthenelse{\equal{#1}{}}{}{(#1)}}

\newcommand{\barc}[1]{\mathcal B\ifthenelse{\equal{#1}{}}{}{(#1)}}

\newcommand{\C}[0]{\mathbf C}
\newcommand{\D}[0]{\mathbf D}

\newcommand{\F}[0]{F}
\newcommand{\G}[0]{G}
\newcommand{\Hcal}[0]{K}
\newcommand{\I}[0]{\mathbf I}

\newcommand{\M}[0]{\mathcal M}

\newcommand{\X}[0]{\mathcal X}
\newcommand{\Y}[0]{\mathcal Y}

\newcommand{\Nbb}[0]{\mathbb N}

\newcommand{\Z}[0]{\mathbb Z}

\newcommand{\htp}{\simeq}

\newcommand{\Id}{\mathrm{Id}}
\newcommand{\OurPoset}{J}

\title[Stability of 2-parameter persistent homology]{Stability of 2-parameter persistent homology}
\author{Andrew J. Blumberg}
\address{Department of Mathematics, University of Texas at Austin, USA}
\email{blumberg@math.utexas.edu}
\thanks{}
\author{Michael Lesnick}
\address{Department of Mathematics and  Statistics, SUNY Albany, Albany, USA}
\email{mlesnick@albany.edu}

\date{}

\begin{document}

\begin{abstract}
The \v Cech and Rips constructions of
persistent homology are stable with respect to perturbations of
the input data.  However, neither is robust to outliers, and both can be insensitive to 
topological structure of high-density regions of the data.
A natural solution is to consider 2-parameter persistence.  This paper studies the stability of 2-parameter persistent
homology: We show that several
related density-sensitive constructions of bifiltrations from data 
satisfy stability properties accommodating the addition and removal of
outliers.  Specifically, we consider the multicover bifiltration, Sheehy's subdivision bifiltrations, and the degree 
bifiltrations. For the multicover and subdivision bifiltrations, we get 1-Lipschitz stability results
closely analogous to the standard stability results for 1-parameter
persistent homology.  Our results for the degree bifiltrations are weaker, but they are tight, in a sense.  As an application of our theory, we prove a law of large
numbers for subdivision bifiltrations of random data.
\end{abstract}
\maketitle

\section{Introduction}\label{sec:intro}

\subsection{Persistent Homology}\label{Sec:Intro_PH}
Topological data analysis (TDA) provides descriptors of the shape of a
data set by first constructing a diagram of topological spaces from
the data and then applying standard invariants from algebraic topology to this diagram.  The most common version of this data analysis pipeline is \emph{persistent homology}; this takes the diagram of spaces to be a \emph{filtration}, 
i.e., a functor 
\[
\F \colon R\to \Top, 
\]
where $R$ is a totally ordered set (regarded as a category in the usual way), such that if $r\leq s$, then $F_r\subset F_s$ and the map $F_{r,s}:F_r\to F_s$ is the inclusion.

Let $H_i$ denote the 
  $i^{\mathrm{th}}$ homology functor with coefficients in a field $K$ and let $\Vect$ denote the category of $K$-vector spaces.  Composition yields a functor 
\[
H_i\F \colon R \to \Vect.
\]
A fundamental structure theorem \cite[Theorem 1.2]{botnan2020decomposition} (see also \cite{crawley2012decomposition,zomorodian2005computing, webb1985decomposition}) tells us that if $M \colon R
\to \Vect$ is a functor with each $M_r$ finite-dimensional, then $M$ is determined up to natural isomorphism by a \emph{barcode} $\B{}_M$, i.e., a multiset of intervals in $R$; 
each interval of $\B{}_M$ corresponds to an indecomposable in a direct sum
decomposition of $M$.  Thus, provided each vector space of $H_i\F$ is
finite-dimensional, $\B{}_{H_i\F}$ provides a well-defined descriptor of the shape of our data set.  We will write $\B{}_{H_i\F}$ simply as $\B{}_i(\F)$.

The choice of the filtration $\F$ depends on the type of data we are
analyzing and the kind of geometric information about the data we wish
to capture.  When our data set $X$ is a set of points in some ambient
metric space $\AmbientSpace$ (e.g., $\AmbientSpace=\R^n$ with the
Euclidean metric), a common choice is the \emph{offset filtration}
$\Offset(X)\colon (0,\infty)\to \Top$, given by 
\[\Offset(X)_r=\bigcup_{x\in X} \Ball(x,r)\] where $\Ball(x,r)$
denotes the open ball in $\AmbientSpace$ with center
$x$ and radius $r$.  

Let $\Simp$ denote the category of simplicial complexes, which we regard as a subcategory of $\Top$ via geometric realization.   According to the \emph{persistent nerve theorem}
\cite{chazal2008towards,bauer2022unified}, an extension of the usual nerve theorem to diagrams of spaces, 
if finite intersections of balls in $\AmbientSpace$ are either empty or contractible, 
then $\Offset(X)$ has the same persistent homology
as the \emph{\v Cech filtration} $\Cech(X)\colon (0,\infty)\to \Simp$, defined by taking $\Cech(X)_r$ to be the nerve of the collection of balls $\{\Ball(x,r)\mid x\in X\}$.  We call $\Cech(X)_r$ a \emph{\v Cech complex of $X$}.

When our data set is a metric space $(X,\partial_X)$ (not necessarily equipped with an embedding into some ambient space), another common choice of filtration is the  \emph{(Vietoris--)Rips filtration} $\VR(X)\colon (0,\infty)\to \Simp$,  defined by taking $\VR(X)_r$ to be the clique complex on the undirected graph with vertex set $X$ and edge set $\{[x,y]\mid x\ne y,\ \partial_X (x,y)<2r\}$.  We call $\VR(X)_r$ a \emph{(Vietoris--)Rips complex of $X$.}

\subsection{Stability of Persistent Homology and its Limitations}\label{Sec:Intro_Stability}
The stability theory for persistent homology tells us that for both the \v
Cech and Rips constructions of persistent homology, small perturbations of the data lead to
correspondingly small perturbations of the barcodes.  The precise statements require three
definitions: 
\begin{itemize}
\item The Hausdorff distance $d_H$ is a metic on non-empty subsets of a fixed ambient metric space, defined by
\[d_H(X,Y)=\inf \{\delta\geq 0\mid X\subset \Offset(Y)_{\delta}\textup{ and } Y\subset \Offset(X)_{\delta}\};\]
\item The Gromov--Hausdorff distance $d_{GH}$, an adaptation of $d_H$
  to arbitrary metric spaces, is defined by
  taking \[d_{GH}(X,Y)=\inf_{\varphi,\psi} d_H(\varphi(X),\psi(Y)),\]
  where $\varphi\colon X\to Z$ and $\psi\colon Y\to Z$ range over all possible isometric embeddings of $X$ and $Y$ into a common metric space $Z$. 
\item The bottleneck distance $d_B$ is a metric on barcodes; roughly, $d_B(\mathcal C,\mathcal D)$  is the minimum, over all matchings between $\mathcal C$ and $\mathcal D$, of the largest distance between the endpoints of matched intervals; see e.g., \cite{cohen2007stability} or \cite{bauer2015induced} for a precise definition.

\end{itemize}
 \begin{theorem}[Stability of persistent homology \cite{cohen2007stability,chazal2009gromov,chazal2014persistence}]\label{Thm:Classical_Stability}
 \mbox{}
\begin{enumerate}[(i)]
\item For any finite $X,Y\subset \R^n$ and $i\geq 0$, \[d_B(\B{}_i(\Cech(X)),\B{}_i(\Cech(Y))) \leq d_{H}(X,Y).\]
\item For any finite metric spaces $X$ and $Y$, and $i\geq 0$, \[d_B(\B{}_i(\VR(X))), \B{}_i(\VR(Y)) \leq d_{GH}(X,Y).\]
\end{enumerate}
\end{theorem}
While this theorem is a conerstone of persistence theory, it says nothing about \emph{robustness}, i.e., stability with respect to outliers. 
%This theorem tells us in particular that persistent homology is stable to metric perturbations with small sup-norm.  However, both
In fact, both \v Cech and Rips persistent homology are notoriously unstable to outliers~\cite[\S
  4]{blumberg2014robust}.
A related issue is that both the \v Cech and Rips constructions have trouble with varying
sampling density, and in particular can be insensitive to topological structure in
high-density regions of the data.  For an illustration of this in the
case of Rips filtrations, see \cite[Figure
  2]{lesnick2015interactive}.

Several strategies have been proposed to address these issues
 within the framework of 1-parameter persistent homology~\cite{carlsson2008local,chazal2009analysis,chazal2013clustering,
  bobrowski2017topological,chazal2011geometric, chazal2018robust,
  phillips2015geometric,blumberg2014robust}; we give an overview in \cref{Sec:Related_Work}.  However, these approaches share certain disadvantages: Unlike the \v Cech and Rips constructions of persistent homology, each requires a choice of one or two parameters; typically, the parameter fixes a spatial scale or a density threshold at which the construction is carried out.  The suitability of a particular choice depends on the data, and a priori, it may not be clear how to make the choice.  In fact, if the data exhibits interesting features at a range of scales or densities, it can be that no single choice of parameter suffices to detect these features.  Moreover, strategies which fix a scale parameter do not distinguish small features in the data from large features, and approaches which fix a density threshold may not distinguish features appearing at high densities from those appearing at low densities.  

All of this suggests that to handle data with outliers or variations in density, it may be advantageous to work with two-parameter analogues of filtrations, called \emph{bifiltrations}, where one of the parameters is a scale parameter, as in the Rips or \v Cech constructions, and the other parameter is a density threshold. 

\subsection{Density-Sensitive Bifiltrations of Point Cloud Data}\label{Sec:Density_Sensitive_Bifiltrations}
 The idea of using multi-parameter filtrations for TDA first appeared in the work of Frosini  and Mullazani, which considered multi-parameter persistent homotopy groups \cite{frosini1999size}, and later, in the work of Carlsson and Zomorodian  \cite{carlsson2009theory}, which introduced multi-parameter persistent homology.  Here, we focus exclusively on the 2-parameter case.

\begin{definition}\label{Def:Bifiltration}
Let $R$ and $S$ be totally ordered sets, and $R\times S$ their product
poset (i.e., $(r_1,r_2)\leq (s_1,s_2)$ if and only if $r_1\leq s_1$ and $r_2\leq
s_2$)).  A \emph{bifiltration} is a functor $\F\colon R\times S\to \Top$ such that if $r\leq s\in R\times S$, then $F_r\subset F_s$ and  $F_{r,s}:F_r\to F_s$ is the inclusion.
\end{definition}

Applying homology to a bifiltration $\F$ by composition yields a functor 
\[
H_i \F\colon R\times S\to \Vect,
\]  
which we call a \emph{bipersistence module}.  A key difficulty with working with multi-parameter
persistent homology is that outside of very special cases, no good
definition of the barcode of a bipersistence module is available;
classical quiver representation theory shows that the natural
analogoue of a barcode for bipersistence modules is an object that is
far too complex to use directly in any data analysis application~\cite{carlsson2009theory}.
Nevertheless, one can define simple invariants of bipersistence modules
 that serve as a surrogate for the barcode; various
proposals for this can be found, e.g.,
in~\cite{carlsson2009theory,lesnick2015theory,cerri2013betti,vipond2018multiparameter,harrington2017stratifying}.

There are several ways to construct a bifiltration $\F$ from metric data in a way that is sensitive to both scale and density.  Given a density function $f\colon X\to [0,\infty)$ on a metric space $X$ (e.g., a kernel or $k$-nearest neighbors density function), \cite{carlsson2009theory} introduced the bifiltration 
\[\begin{aligned}
&\Rips(f,r)\colon (0,\infty)^{\op}\times \R\to \Top,\\
&\Rips(f,r)_{(k,r)}=\Rips(\{x\in X\mid f(x)\geq k\})_r.
\end{aligned}\]
A  \v Cech version of this bifiltration $\Cech(f,r)$ can be defined in the same way.  We will refer to these as \emph{density bifiltrations}.  In general, a density function depends on a choice of bandwidth parameter.  Thus, a density bifiltration depends on a parameter choice.  
This is arguably a disadvantage of the construction.

In this paper, we study density-sensitive bifiltrations built from metric data whose definitions do not involve a parameter choice.  Specifically, we consider the following bifiltrations; see \cref{sec:bifilt} for the formal definitions.
\begin{itemize}
 \item The \emph{multicover} bifiltration $\nMult(X)$ of a set of points $X$ in a metric space $\AmbientSpace$ is a 2-parameter extension of the offset filtration which takes into account the number of times a point in $\AmbientSpace$ is covered by a ball~\cite{sheehy2012multicover,edelsbrunner2018multi,chazal2011geometric}. 
\item Sheehy introduced density-sensitive extensions of the  \v Cech and Rips filtrations, which we call the \emph{subdivision-\v Cech
and subdivision-Rips bifiltrations}, and denote by $\nSCech(-)$ and $\nSRips(-)$ \cite{sheehy2012multicover}.  The definition of these amounts to the observation that there is a natural filtration on the barycentric subdivision of any simplicial complex.  
\item The \emph{degree-Rips} bifiltration $\nDRips(-)$, another density-sensitive extension of the Rips filtration, was introduced by Lesnick and Wright \cite{lesnick2015interactive} and has been studied in several works \cite{mcinnes2017accelerated,jardine2019stable,jardine2020persistent,jardine2019data,jardine2017cluster,rolle2020stable}.  This bifiltration has a natural \v Cech analogue, the \emph{degree-\v Cech} bifiltration $\nDCech(-)$.  The definition of these bifiltrations amounts to the observation that any simplicial complex is naturally filtered by vertex degree.  
\end{itemize}
We discuss the computation of these bifiltrations in \cref{Sec:Computation}.  

\begin{definition}\label{Def:Good_MS}
We say a metric space is \emph{good} if the intersection of any
finite collection of open balls is either empty or contractible.
\end{definition}

%The primary example of a good metric space is $\R^n$ with the $\ell^p$-metric, for some $1\leq p\leq \infty$.   %Aside from the standard case $p=2$, the case $p=\infty$ is the most important to us here, because it arises in the proof of \cref{Thm:Classical_Stability}\,(ii) appearing in \cite{chazal2009gromov}, and also, in essentially in the same way, in our stability proofs for the subdivision-Rips and degree-Rips bifiltrations.

According to a 2-parameter extension of Sheehy's multicover nerve theorem \cite{sheehy2012multicover}, 
the multicover bifiltration is in fact topologically equivalent (i.e., weakly equivalent, see \cref{Def:Weakly_Equivalent}) to the
subdivision-\v Cech bifiltration whenever the ambient metric space is good.  See \cref{Thm:Multicover} for the
precise statement of the theorem, and \cref{Sec:Multicover_Nerve} for a proof (following Cavanna, Gardner, and
Sheehy \cite{cavanna17when}) and further discussion.  

\subsection{Our Results}\label{Sec:Our_Results}
In this paper, we present stability results for all of the
parameter-free constructions of bifiltrations mentioned above.  These
results are formulated using metrics on data sets which are robust to
outliers.  As an application of our results, we prove a law of large numbers for the
subdivision-\v Cech bifiltrations of random samples of a metric measure space.

For the multicover bifiltration and the subdivision bifiltrations, we
give 1-Lipschitz stability results closely analogous to the bounds of
\cref{Thm:Classical_Stability}.   
On the other hand, we show that analogous Lipschitz stability results are not possible for the degree
bifiltrations; these bifiltrations are only stable in a weaker sense.

To state our stability results, we need to first specify the distances on
bipersistence modules and on data sets that appear in the statements.
Given the lack of barcodes for bipersistence modules, when
developing stability or inference theory for multi-parameter
persistent homology, a natural approach is to formulate results
directly on the level of persistence modules
\cite{lesnick2015theory}. In fact, we take a stronger approach here, and formulate results directly on
the level of bifiltrations, using the language of \emph{homotopy
  interleavings}~\cite{blumberg2017universality}.  Interleavings are a standard formalism in
TDA for quantifying the similarity between
diagrams of topological spaces or vector spaces.  
Homotopy interleavings are a variant which satisfy a homotopy
invariance property.

Both notions of interleaving give rise to metrics, the \emph{interleaving
distance} $d_I$ and the \emph{homotopy interleaving distance} $d_{HI}$.  By
applying the homology functor, our homotopy interleaving results yield
results on the level of persistence modules as corollaries, given in
terms of (strict) interleavings.

To give a metric on data sets, we regard a data set as a metric
measure space.  Our main results are formulated using the \emph{Prohorov distance} $\Pro$, and the
\emph{Gromov--Prohorov distance} $\GPr$~\cite{greven2009convergence,janson2020gromov}.  The Prohorov distance is
defined between two probability measures on a fixed metric space; it is a classical tool in probability theory, commonly used to metrize
weak convergence.  The Gromov--Prohorov distance is a natural adaptation to measures defined on different metric spaces.  %
The Prohorov distance can be thought of as an analogue of the
Hausdorff distance for measures \cite[\S 27]{Villani2008}.  Our
results demonstrate that this analogy extends to a stability theory
for 2-parameter persistent homology closely paralleling the one for
1-parameter persistence.

\begin{remark}\label{Rem:Wasserstein}
The \emph{Wasserstein distance}~\cite[\S 6]{Villani2008} is another classical distance on probability measures which has played an important role in prior work on robust TDA~\cite{chazal2011geometric, phillips2015geometric,buchet2016efficient}; see \cref{Sec:Related_Work}.  As with the Prohorov distance, we have a variant of the Wasserstein distance for probability measures defined on different metric spaces, the \emph{Gromov--Wasserstein distance} \cite{Memoli11}.  The Prohorov distance admits a simple upper bound in terms of the Wasserstein distance \cite{bjerkevik2020presentation}, which implies an analogous bound relating the Gromov--Prohorov and Gromov--Wasserstein distances; see \cref{Prop:Prohorov_Wasserstein_Bounds}.  Hence, the stability bounds we give in terms of the (Gromov--)Prohorov distance (\cref{Thm:Multicover_and_Subdivision_Cech_Stability,Thm:Degree_Rips_Stab} below) yield corresponding bounds in terms of the (Gromov--)Wasserstein distance. 
\end{remark}

We now turn to the statements of our main results.  

\begin{notation}\label{Not:Counting_Measures}
For $X$ a non-empty finite metric space, let $\mu_X$ denote the
uniform probability measure on $X$, i.e., $\mu_X(A)=|A|/|X|$ for all
sets $A\subset X$.  If $X$ is equipped with an embedding $j\colon X\hookrightarrow \AmbientSpace$ into some ambient metric space $\AmbientSpace$, then let $\nu_X$ denote the pushforward measure $j_*(\mu_X)$, i.e., $\nu_X(A)=|A\cap X|/|X|$ for all Borel sets $A\subset \AmbientSpace$.
\end{notation}

\begin{theorem}\label{Thm:Multicover_and_Subdivision_Cech_Stability}\label{Cor:Degree_Rips_Stab}
\mbox{}
\begin{enumerate}
\item[(i)] For $X$ and $Y$ non-empty, finite subsets of a metric
  space $\AmbientSpace$, \[d_{I}(\nMult(X),\nMult(Y)) \leq \Pro(\nu_X,\nu_Y).\]
\item[(ii)] In (i), if we assume further that $\AmbientSpace$ is good, then also \[d_{HI}(\nSCech(X),\nSCech(Y)) \leq \Pro(\nu_X,\nu_Y).\]
\item[(iii)] For any non-empty, finite metric spaces $X$ and $Y$, 
\[
d_{HI}(\nSRips(X),\nSRips(Y)) \leq \GPr(\mu_X,\mu_Y).
\] 
\end{enumerate}
\end{theorem}

\cref{Thm:Multicover_and_Subdivision_Cech_Stability}\,(ii) and (iii)
are close analogues of \cref{Thm:Classical_Stability}\,(i) and
(ii), respectively, and have closely analogous proofs.  The proof
of \cref{Thm:Multicover_and_Subdivision_Cech_Stability}\,(i), given
in \cref{Sec:_Sbdivision_Stability}, is short and straightforward.  In
fact, we prove a generalization which holds for arbitrary measures on
a common metric space (\cref{Thm:Measure_Stability}).   
\cref{Thm:Multicover_and_Subdivision_Cech_Stability}\,(ii) follows
from \cref{Thm:Multicover_and_Subdivision_Cech_Stability}\,(i) via the  
2-parameter extension of Sheehy's multicover nerve theorem (\cref{Thm:Multicover}).

\cref{Cor:Degree_Rips_Stab}\,(iii) follows
readily from \cref{Cor:Degree_Rips_Stab}\,(ii), using an embedding
argument similar to the one used to prove
\cref{Thm:Classical_Stability}\,(ii) in \cite{chazal2009gromov}.

Our stability results for the degree bifiltrations are formulated in
terms of \emph{generalized interleavings}
\cite{lesnick2015theory,bubenik2015metrics}, where the shift maps are
affine maps rather than the usual translations.  Specifically, for $\delta\geq 0$, we say
that two bifiltrations are $\gamma^\delta$ interleaved if they are
interleaved with respect to the affine map
\[
(x,y) \mapsto (x-\delta,3y+\delta);
\]
see \cref{Sec:Interleavings} for the formal definition.

\begin{theorem}[Stability of degree bifiltrations]
  \label{Thm:Degree_Rips_Stab}  
    \mbox{}
\begin{enumerate}
\item[(i)] If $X$ and $Y$ are non-empty finite subsets of a good metric space, then
  $\nDCech(X)$ and $\nDCech(Y)$ are $\gamma^\delta$-homotopy
  interleaved for all $\delta>\Pro(\nu_X,\nu_Y),$
\item[(ii)] If $X$ and $Y$ are non-empty finite metric spaces, then  $\nDRips(X)$ and $\nDRips(Y)$ are
  $\gamma^\delta$-homotopy interleaved for all $\delta>\GPr(\mu_X,\mu_Y)$.
\end{enumerate}
\end{theorem}

To prove \cref{Thm:Degree_Rips_Stab}, we show that there is a simple
linear interleaving between the subdivision and degree filtrations of
a given data set (\cref{Prop:Degree_Cech_vs_Subdivision_Cech}); in
this sense, degree bifiltrations are approximations of subdivision
bifiltrations, in much the same way that Rips filtrations of
point clouds are approximations of \v Cech filtrations.
\cref{Thm:Degree_Rips_Stab} follows from this and the stability of the
subdivision filtrations, using a ``generalized triangle inequality''
for homotopy interleavings.  

We show that \cref{Thm:Degree_Rips_Stab} is tight, in the
sense that if the 
constant $3$ appearing in the definition of $\gamma^{\delta}$ is made
any smaller, then neither statement of the theorem holds; see \cref{Prop:Degree_Tightness}.  We also
observe in \cref{Rmk:Discontinuity} that the degree filtrations are
discontinuous with respect to the relevant metrics on point clouds and bifiltrations.  Thus, \cref{Thm:Degree_Rips_Stab} is far weaker than the corresponding
results for subdivision bifiltrations.

Our final theoretical result, a law of large numbers for subdivision-\v Cech bifiltrations, is an application of the stability theory
developed here: We show that for $\X$ any separable metric probability space,  
the subdivision-\v Cech bifiltration of an i.i.d. sample of $\X$
converges almost surely in the homotopy interleaving distance to a
bifiltration constructed directly from $\X$, as the sample size tends
to $\infty$.  \cref{Thm:Cech_Convergence} gives the precise statement.

\begin{remark}
\cref{Thm:Multicover_and_Subdivision_Cech_Stability,Thm:Degree_Rips_Stab} also admit ``unnormalized'' variants which make no finiteness
assumptions on the data; see
\cref{Rem:Unnormalized_Results}. 
\end{remark}

\begin{remark}
In this paper, we do not address the stability of the density bifiltrations $\Rips(f,r)$ and $\Cech(f,r)$ defined above.  We expect that the approach to stability considered here can be adapted to yield results about density bifiltrations.  
\end{remark}

\subsection{Computational Study of the Stability of Degree-Rips Bifiltrations}
As discussed in \cref{Sec:Computation} below, degree-Rips bifiltrations are known to be computable in practice.  In work to be reported elsewhere, we have found them to be a very convenient tool for studying cluster structure in genomic data.  Thus, their stability with respect to outliers is of particular interest to us.  

To get a sense of what our stability result for degree bifiltrations (\cref{Thm:Degree_Rips_Stab}) may mean for practical data analysis, in \cref{Sec:Examples} we use the 2-parameter persistence
software RIVET \cite{lesnick2015interactive,lesnick2019computing,rivet} to study the degree-Rips bifiltrations of three synthetic data sets, denoted $X$, $Y$, and $Z$.  $X$ is a uniform sample of 475 points from an annulus; $Y$ is obtained from $X$ by adding 25 points sampled uniformly at random from the disc bounded by the annulus; and $Z$ is a uniform sample of 500 points from a disk.  

We find that in spite of the strong topological signal shared by $X$ and $Y$, \cref{Thm:Degree_Rips_Stab} yields a trivial constraint between the $1^{\mathrm{st}}$ persistent homology modules of the degree-Rips bifiltrations of these two data sets, and thus offers no information about the similarity of the two modules (\cref{Rem:Symmetric_Version_No_Good}).  This suggests that \cref{Thm:Degree_Rips_Stab} may be of limited use in the analysis of real-world data with outliers, despite being tight in the sense explained above.
However, we show that when one data set is a subset of the other, \cref{Thm:Degree_Rips_Stab} admits an asymmetric variant giving substantially tighter bounds (\cref{Prop:Degree_Subspaces}), and this result does constrain the relationship between the persistent homology modules of $X$ and $Y$ (\cref{Sec:Stability_Analysis}).  This demonstrates that our theory does provide nontrivial information about the robustness of degree bifiltrations.

Most interestingly, the homology modules of $X$ and $Y$ exhibit evident structural similarities that
are not explained by the available stability theory (\cref{Rem:Stronger_Stability}).  
This suggests that in future work, one should seek a more refined stability theory for degree
bifiltrations that can explain these similarities.  For this, it might be fruitful to use the framework of homology inference, as considered in \cite{oudot2015zigzag}.

\subsection{Computing the Multicover, Subdivision, and Degree bifiltrations}\label{Sec:Computation}
For practical applications of the bifiltrations studied in this paper, computability is a critical consideration.  We now discuss what is known about computing these bifiltrations.

While interesting theoretically, the subdivision bifiltrations $\nSCech(X)$ and $\nSRips(X)$ of a data set $X$ are usually too large to be directly computed in practice; the largest simplicial complex in each of these bifiltrations has exponentially many vertices in the size of $X$.

The degree bifiltrations are much smaller than the corresponding subdivision bifiltrations: For fixed $k$, the $k$-dimensional skeleton of the degree-Rips bifiltration of a metric space $X$ has size $O(|X|^{k+2})$.  Moreover, if the bifiltration is coarsened to lie on a grid of constant size, then the $k$-skeleton has size $O(|X|^{k+1})$, which agrees (asymptotically) with that of the ordinary Rips filtration.  Using a line sweep algorithm designed by Roy Zhao  \cite{lesnick2021computing} and implemented in RIVET, these low-dimensional skeleta are readily computed in practice for data sets with hundreds of points.

Though the subdivision-\v Cech bifiltration is too large to compute directly, recent work offers hope that practical
computations of multicover persistent homology may be within reach for data lying in a low-dimensional Euclidean space: It is shown in~\cite{corbet2021computing} that for $X\subset \R^n$ finite, a polyhedral bifiltration in $\R^{n+1}$ called the \emph{rhomboid bifiltration}, introduced by Edeslbrunner and Osang \cite{edelsbrunner2021multi}, is weakly equivalent to a version the multicover bifiltration of $X$ \cite{corbet2021computing}.  For fixed $n$, the rhomboid bifiltration has size $O(|X|^{n+1})$.   A recent algorithm of Edelsbrunner and Osang \cite{EdOs20} computes the rhomboid bfiltration.  The algorithm and its complexity analysis depend on a choice of algorithm for computing weighted Delaunay bifiltrations.  For a suitable such choice, the algorithm computes the rhomboid bifiltration of $X\subset \R^3$ in time $O(|X|^5)$; see~\cite[\S 4.5]{corbet2021computing}.  The algorithm has been implemented \cite{osang2020rhomboid}, and scales well enough to compute the full rhomboid bifiltration of at least 200 points in $\R^3$; examples and timing results appear in \cite{corbet2021computing} and \cite{EdOs20}.  

While the rhomboid bifiltration will not be discussed elsewhere in this paper, our main results about multicover and subdivision-\v Cech filtrations also hold for rhomboid bifiltrations, since these results are formulated in a homotopy-invariant way.

No efficient algorithms are known for computing or approximating (up to homotopy) either the multicover bifiltration of high-dimensional data or the subdivision-Rips bifiltration.  The development of such algorithms seems to be an important direction for future research.   

\subsection{Other Related Work}\label{Sec:Related_Work}

As noted in \cref{Sec:Intro_Stability}, several strategies have been proposed to address issues of
robustness and density within the framework of 1-parameter persistent homology; we now give an overview.  Early work
of Carlsson et al.\ on the application of persistent homology to
natural scene statistics \cite{carlsson2008local} dealt with outliers by explicitly removing
them.  Specifically, the authors considered a
density function $f\colon X\to (0,\infty)$ on the data $X$ and cleaned the
data by removing points whose density value was below some fixed threshold $t>0$.  The persistent homology of data preprocessed in this way is easily seen to be unstable with respect to perturbations of the data, the density threshold $t$, and the bandwidth parameter.  
  
An alternative approach is to work with the 1-parameter filtrations obtained by fixing the scale parameter in a density bifiltration, while allowing the density threshold to vary.  Chazal et al.\ \cite{chazal2009analysis,chazal2013clustering} considered the use of such filtrations in the topological analysis of density functions and in clustering.  
It follows from the results of this work that these filtrations can be used to consistently 
  estimate the superlevel persistent homology of a probability density
  function, provided the scale parameter and bandwidth of the density function are
  chosen in the right way; see also \cite[Chapter
    4]{lesnick2012multidimensional}.  Subsequent work of Bobrowski, Mukherjee, and Taylor ~\cite{bobrowski2017topological} revisited this estimation problem, using
  different estimators. 

For a finite set of points in some ambient metric space, Chazal, Cohen-Steiner, and M\' erigot introduced a novel variant of the offset filtration, the \emph{distance-to-measure} filtration, which is robust to outliers~\cite{chazal2011geometric,
  chazal2018robust,buchet2016efficient}.
More specifically, its persistent homology is stable with
respect to the 2-Wasserstein distance on the 
data and the bottleneck distance on barcodes.  The distance-to-measure filtration depends on a parameter which, roughly
speaking, controls the scale on which smoothing is performed in the
construction.  For sufficiently small values of this parameter, the distance-to-measure filtration is equal to the offset filtration.

For Euclidean data, Phillips, Wang, and Zheng~\cite{phillips2015geometric} introduced an alternative construction, the
\emph{kernel distance filtration}, which exhibits similar robustness
properties.  This filtration is
determined by a kernel density function, and so depends on a choice of
bandwidth parameter.  Subsequently,~\cite{bobrowski2017topological} directly studied the superlevel persistent homology of a
kernel density function and its
statistical properties.

Blumberg et al.~\cite{blumberg2014robust} studied distributions of barcodes induced by samples of a fixed size drawn from metric measure
spaces.  The authors proved that their invariants satisfy a stability result with respect to the Gromov--Prohorov
distance on metric measure spaces and the Prohorov distance on
distributions.  This implies that the invariants are robust to outliers.  These invariants depend on a choice of 
fixed size for the subsamples, which reflects the local feature scale
of the data.

In the multiparameter setting, Rolle and Scoccola \cite{rolle2020stable}, \cite[\S 6.5]{scoccola2020locally} have recently shown that if
instead of using the Prohorov distance to formulate
\cref{Thm:Degree_Rips_Stab}, one uses the maximum of the Prohorov and
Hausdorff distances (the Prohorov-Hausdorff metric), then the bounds
of \cref{Thm:Degree_Rips_Stab} strengthen to Lipschitz bounds.  In fact, their result is more general, concerning 3-parameter Rips filtrations associated to metric probability spaces.  
Rolle and Scoccola apply this result to prove the stability and consistency of a clustering scheme.  
However, the Prohorov-Hausdorff metric used to formulate these stability results is not robust to outliers.  

In the case where one data set is a subset of the other, recent work of Jardine \cite{jardine2020persistent} presents another approach to the stability of degree-Rips bifiltrations.  This approach yields a stability theorem phrased in terms of homotopy commutative
interleavings (a weaker notion than the homotopy coherent
interleavings used here and in~\cite{rolle2020stable}) and a metric on
data sets which is bounded below by the Prohorov-Hausdorff metric.

\subsection{Outline}
The paper is organized as follows.  In~\cref{sec:background},
we give a concise review of necessary background on metric measure
spaces and 2-parameter persistence.  In~\cref{sec:theorems}, we
prove the main stability results.  Some of these depend on the multicover
nerve theorem for bifiltrations, which we prove in~\cref{Sec:Multicover_Nerve}, following~\cite{cavanna17when}.  \cref{Sec:RIVET_Point_Clouds} presents our computational study of the stability of degree-Rips bifiltrations.  \cref{sec:tightness} gives the proof of \cref{Prop:Degree_Tightness}, which says that our stability result for degree bifiltrations (\cref{Thm:Degree_Rips_Stab}) is tight.

\subsection*{Acknowledgements}

We thank Mike Mandell for his insights on various aspects of robustness related to
this paper; Ren\' e Corbet, Alex Rolle, and Don Sheehy for helpful conversations about
the multicover nerve theorem; H\aa vard Bjerkevik for valuable 
discussions about the Wasserstein stability of 2-parameter
persistence; Alex Tchernev for pointing out an error in \cref{Sec:Interleavings} of the first version of the paper; and the anonymous reviewers for many helpful  suggestions.

The computations and figures of
\cref{Sec:RIVET_Point_Clouds} would not have been possible without the
work of Matthew Wright, Bryn Keller, Roy Zhao, and Simon Segert on the
RIVET software.  In particular, Zhao designed and implemented RIVET's
algorithm for computing degree-Rips bifiltrations, and Segert made
critical improvements to RIVET's visualization capabilities. 

The
development of RIVET was supported in part by NSF grant DMS-1606967.
Blumberg was partially supported by NIH grants 5U54CA193313 and
GG010211-R01-HIV, AFOSR grant FA9550-18-1-0415, and NSF grant CNS
1514422. 

\section{Background}\label{sec:background}

In this section, we review the definitions and concepts used in the
rest of the paper.

\subsection{Rips and \v Cech complexes}
We begin with a remark on the definitions of \v Cech and Rips
complexes given in \cref{Sec:Intro_PH}.  Let us recall these: 

\begin{definition}\label{Def:Cech}
Given a subset $X$ of a metric space $\AmbientSpace$ and $r>0$, the \emph{\v
Cech complex} $\Cech(X)_r$ is the  
nerve of the collection of open balls $\{\Ball(x,r)\mid x\in X\}$ in $\AmbientSpace$.
\end{definition}

\begin{definition}\label{Def:Rips}
Given a metric space $(X, \partial_X)$ and $r>0$, the \emph{Rips
complex} $\VR(X)_r$ is the clique complex on the graph
with vertex set $X$ and edge set $\{[x,y]\mid x\ne y,\ \partial_X(x,y)<2r\}.$
\end{definition}

\begin{remark}\label{Rem:Different_Def_Of_Rips_Cech}
These definitions are the ones used
in~\cite{chazal2009gromov}.  Many references use 
a slightly different pair of definitions, where ``open'' is replaced with ``closed'' in \cref{Def:Cech}, and 
the strict inequality $<$ is replaced with $\leq$
in \cref{Def:Rips};  we call the definitions we have given the \emph{open convention}, and the alternative definitions the \emph{closed convention}.  
All stability results for Rips and \v Cech
(bi)filtrations appearing in this paper hold exactly as written if we instead use the closed convention. 

The closed convention is arguably more natural, since it yields finitely presented persistence modules and is more convenient when studying rhomboid bifiltrations.  
However, the open convention is convenient for interfacing with the multicover nerve theorem, which (in the formulation we consider) concerns open covers.  At the time this paper was written, no version of the multicover nerve theorem was available for closed covers.  But using ideas from \cite{bauer2022unified}, which was released after our paper was finished, it is possible to establish the multicover nerve theorem for a large class of closed covers; see \cref{Rem:Multicover_Closed}.  Had \cite{bauer2022unified} been written before our paper was finished, we might have chosen to work with the closed convention throughout.    

We also note that in some references, the constant 2 in our definition of the edge set of $\VR(X)$ is replaced with 1.   Our convention ensures that $\VR(X)_r$ is the clique complex of the 1-skeleton of $\Cech(X)_r$.
\end{remark}

\subsection{Metric Measure Spaces}\label{Sec:Metric_Measure_Spaces}

The following definition allows us to formalize the notion of sampling from a metric space.

\begin{definition}\label{Def:MM_Space}
A \emph{metric measure space} is a triple $(\X,\partial_\X,\eta_\X)$, where
$(\X,\partial_\X)$ is a metric space and $\eta_\X$ is a measure on its
Borel $\sigma$-algebra.  We will sometimes abuse notation slightly and
denote such a triple as either $\X$ or $\eta_\X$.
\end{definition}

We will often work with \emph{metric probability spaces}, i.e., metric
measure spaces where $\eta_\X(\X) = 1$.  Of course, any metric measure
space with finite measure can be normalized by dividing the measure of
each subset by $\eta_\X(\X)$, yielding a metric probability space.

\subsection{Bifiltrations from Data}\label{sec:bifilt}

Recall that in \cref{Def:Bifiltration}, we have defined a bifiltration to be a functor 
$F\colon R\times S\to \Top$ for some totally ordered sets $R$ and $S$,
such that $F_{a}\subset F_{b}$ whenever $a\leq b\in R\times S$.  For
$P$ a poset, let $P^{\op}$ denote the opposite poset.  In this paper,
we will always take $R=(0,\infty)^{\op}$ and $S=(0,\infty)$, except in
the computational example of \cref{Sec:RIVET_Point_Clouds}.  We let
$\OurPoset=(0,\infty)^{\op}\times (0,\infty)$. 

\begin{remark}
It is common in the literature on multiparameter persistence to see bifiltrations indexed by other choices of $R\times S$, e.g., by $\R^2$ or $\mathbb{N}^2$.  However, for the constructions considered in this paper, $\OurPoset$ is a natural choice of indexing poset.
\end{remark}

We now define the bifiltrations that we will study in this paper.  First, we consider a bifiltration associated to an arbitrary metric measure space.
A variant of the following definition was considered in \cite{chazal2011geometric}.
\begin{definition}\label{Def:Measure_Bifiltration}
For  $\X$ a metric measure space, we define the \emph{measure bifiltration} of $\X$ to be the bifiltration 
\begin{align*}
&\MB(\X)\colon(0,\infty)^{\op}\times (0,\infty) \to \Top,\\
&\MB(\X)_{(k,r)}=\{y\in \X\mid \eta_\X(B(y,r))\geq k\}.
\end{align*}
That is, $y$ is contained in $\MB(\X)_{(k,r)}$ if and only the open ball of
radius $r$ centered at $y$ has measure at least $k$.
\end{definition}

We will be particularly interested in the following special case:

\begin{definition}\label{Def:Multicover}
For $X$ a subset of a metric space
$(\AmbientSpace,\partial_{\AmbientSpace})$, let \[\Mult(X)=\MB(\tilde
\nu_X),\] where $\tilde \nu_X$ is the \emph{counting measure of $X$},
i.e., the measure on $\AmbientSpace$ is given by $\tilde
\nu_X(A)=|A\cap X|$.  Thus,  
\[\Mult(X)_{(k,r)}=\{y\in \AmbientSpace\mid
\partial_{\AmbientSpace}(y,x)< r \textup{ for least $k$ distinct
  points $x\in X$}\}.\]  
  We call $\Mult(X)$ the \emph{(unnormalized) multicover bifiltration} of $X$.
\end{definition}

\begin{remark}
Note that in our our notation, the 
dependence of $\Mult(X)$ on the ambient metric space $\AmbientSpace$ is
implicit.   The superscript $\mathrm{u}$ in the notation is intended to denote that this bifiltration is unnormalized; we will define a normalized version below, which we denote without the $\mathrm{u}$.  We will also use the same convention to denote the unnormalized and normalized versions of the other bifiltrations we consider.
\end{remark}

\begin{remark}
The definition of the multicover bifiltration given above is the one 
used by Sheehy~\cite{sheehy2012multicover}.  Edelsbrunner and Osang~\cite{edelsbrunner2018multi} consider a variant of the
definition using closed rather than open balls; the two variants 
have interleaving distance 0.
\end{remark}

Next, we define the subdivision bifiltrations of Sheehy.  First, we
review the definition of the barycentric subdivision of an (abstract)
simplicial complex.  Given a simplicial complex $S$, a \emph{flag} in
$S$ is a sequence of simplices in $S$ 
\[
\sigma_1\subset \sigma_2\subset \cdots \subset \sigma_m,
\]
where each containment is strict.

\begin{definition}\label{defn:barycentric}
The \emph{barycentric subdivision}
of a simplicial complex $T$, which we denote $\Bary(T)$, is the simplicial complex whose $k$-simplices are the flags of length $k+1$, with the face relation
defined by removal of simplices from a flag. 
\end{definition}

Let $\Simp$ denote the category of simplicial complexes.  There is a
natural filtration on $\Bary(T)$ by dimension, which is functorial on
inclusions: 

\begin{definition}\label{Def:Subdivision_Cech_Rips}\label{defn:sub-filt} \mbox{}
\begin{enumerate}[(i)]
\item For $T$ a simplicial complex, the \emph{subdivision filtration} \cite{sheehy2012multicover}
\[
\Sd(T)\colon (0,\infty)^{\op}\to \Top
\]
is defined by taking  $\Sd(T)_k\subset \Bary(T)$ to be the set of flags
whose minimum element has dimension at least $k-1$.
\item It is easy to check that if $\F:\C\to \Simp$ is a functor whose internal maps are monomorphisms (i.e., injections on vertex sets), then the filtrations
$\{\Sd(\F_r)\}_{r\in (0,\infty)}$ assemble into a
  functor \[\Sd(\F)\colon (0,\infty)^{\op}\times \C\to \Simp.\]
For $X$ a subset of a metric space $\AmbientSpace$, we call $\Sd(\Cech(X))$ the \emph{subdivision-\v Cech}
bifiltration and denote it $\SCech(X)$.
Similarly, for $X$ a metric space, we call $\Sd(\Rips(X))$ the
\emph{subdivision-Rips} bifiltration, and denote it $\SRips(X)$.
\end{enumerate}
\end{definition}

As noted in \cref{sec:intro}, the subdivision bifiltrations are too
big to construct in practice.  This motivates the consideration of smaller 
density-sensitive simplicial bifiltrations from point cloud
data.  The following generalizes a construction introduced in~\cite{lesnick2015interactive}:

\begin{definition}\mbox{}
\begin{enumerate}[(i)]
\item For $T$ a simplicial complex, the \emph{degree filtration}
\[
\Deg(T)\colon (0,\infty)^{\op}\to \Simp
\]
is defined by taking $\Deg(T)_k$ to be the maximum subcomplex of $T$ whose vertices have degree at least $k-1$ in the 1-skeleton of $T$.  
\item If $\F:\C\to \Simp$ is any functor whose internal maps are monomorphisms, then the filtrations
$\{\Deg(\F_r)\}_{r\in (0,\infty)}$ assemble into a
  functor \[\Deg(\F)\colon (0,\infty)^{\op}\times \C\to \Simp.\]
For $X$ a subset of a metric space $\AmbientSpace$, we call $\Deg(\Cech(X))$ the \emph{degree-\v Cech}
bifiltration and denote it $\DCech(X)$.  For $X$ a metric space, we call $\Deg(\Rips(X))$ the
\emph{degree-Rips} bifiltration, and denote it $\DRips(X)$.
\end{enumerate}
\end{definition}

In the case of non-empty, finite data sets, the multicover,
subdivision, and degree bifiltrations each admit a variant where the
first parameter of the bifiltration is normalized by the number of
points in the data set.  For example: 

\begin{definition}[Normalized Bifiltrations]\label{Def:Norm-Multicover}
For $X$ a non-empty finite subset of a metric space $\AmbientSpace$, define the \emph{(normalized) multicover bifiltration} of $X$ to  be the bifiltration $\nMult(X)$ given by \[\nMult(X)_{(k,r)}=\Mult(X)_{(k|X|,r)}.\]
Equivalently, $\nMult(X)=\MB(\nu_X)$, where as in \cref{Not:Counting_Measures}, $\nu_X$ is the measure on $\AmbientSpace$ is given by
\[\nu_X(A)=|A\cap X|/|X|.\] 
We define normalized variants of the subdivision and degree bifiltrations analogously, and also denote them 
by removing the $\mathrm{u}$.  For example, the normalized subdivision-Rips
filtration of $X$ is denoted $\nSRips(X)$. 
\end{definition}

\subsection{Distances on Metric Measure Spaces}\label{Sec:Distances_on_Measures}
There are many ways to define a distance between probability measures on
a fixed metric space~\cite{Su2002}.  We will focus on two standard
choices, the Prohorov and Wasserstein distances.
To adapt these to distances between measures defined on different
metric spaces, we use Gromov's idea of minimizing over isometric
embeddings into a larger ambient space. 

\begin{definition}The \emph{Prohorov distance} (also known as the Prokhorov distance)
between measures $\mu$ and $\eta$ on a metric space $( \AmbientSpace,\partial_{ \AmbientSpace})$ is given by 
\[
\begin{split}
\Pro(\mu,\eta) = \sup_A \inf\{\delta \geq 0 \mid\ & \mu(A)
\leq \eta(A^\delta) + \delta \textup{ and } \\ 
&\eta(A) \leq \mu(A^\delta) + \delta\},
\end{split}
\] 
where $A \subset  \AmbientSpace$ ranges over all closed sets and 
\[A^\delta=\{y\in  \AmbientSpace\mid \partial_{\AmbientSpace}(y,a) < \delta\textup{ for some }a \in A.\}\]
\end{definition}

\begin{definition}[\cite{greven2009convergence}]
The \emph{Gromov--Prohorov distance} between metric measure spaces $\X=(\X,\partial_\X,\eta_\X)$ and $\Y=(\Y,\partial_\Y,\eta_\Y)$ is
\[
\GPr(\X,\Y) = \inf_{\varphi, \psi} \Pro(\varphi_*(\eta_\X), \psi_*(\eta_\Y)),
\]
where $\varphi \colon \X \to Z$ and $\psi \colon \Y \to Z$ range over all
isometric embeddings into a common metric space $Z$. 
\end{definition}

\begin{remark}[Robustness]\label{Rem:Outliers}
It is easy to see that if $Y$ is a finite metric space and $X\subset Y$ is nonempty, then the
uniform probability measures $\mu_X$ and $\mu_Y$ of $X$ and $Y$ satisfy
\[\GPr(\mu_X,\mu_Y)\leq\Pro(\mu_X,\mu_Y)\leq \frac{|Y\setminus X|}{|X|}.\]
In this sense, $\Pro$ and $\GPr$ are robust to outliers.
\end{remark}

A \emph{coupling} between metric measure spaces $\X$ and $\Y$ is a measure $\mu$ on $\X \times \Y$ such that $\mu(A \times \Y) = \eta_\X(A)$ and $\mu(\X \times B) = \eta_\Y(B)$ for all Borel sets $A \subseteq
\X$ and $B \subseteq \Y$. 

 For the next definition, recall that a metric space is said to be \emph{Polish}
if it is separable (equivalently, second-countable) and complete.

\begin{definition}
For $p\in [1,\infty)$, the \emph{$p$-Wasserstein distance} between probability measures
$\mu_1$ and $\mu_2$ on a Polish metric space $( \AmbientSpace, \partial_{ \AmbientSpace})$ is    
\[
\Wa^p(\mu_1, \mu_2) = \inf_{\mu} \left(\int_{ \AmbientSpace \times  \AmbientSpace}
\partial_{ \AmbientSpace}(y,z)^p\, d\mu \right)^{\frac{1}{p}},
\]
where $\mu$ ranges over all couplings of $\mu_1$ and $\mu_2$.
\end{definition}

\begin{definition}
Let $\X$ and $\Y$ be Polish metric probability spaces.  For $p\in [1,\infty)$, the \emph{$p$-Gromov--Wasserstein distance} between $\X$ and $\Y$ is 
\[
\GW^p(\X,\Y)=\inf_{\varphi,\psi} \Wa^p(\varphi_*(\eta_\X), \psi_*(\eta_\Y)),
\]
where $\varphi \colon \X \to Z$ and $\psi \colon \Y \to Z$ range over all
isometric embeddings into a common metric space $Z$. 
\end{definition}

\subsubsection{Metric Properties}
We next consider the question of when each of the distances on measures we have defined above yields a metric.  We begin with some definitions.  

\begin{definition}\mbox{}
The \emph{support} of a metric measure space $\X=(\X,\partial_\X,\eta_\X)$, denoted $\Supp(\X)$, is the complement of the set 
\[\bigcup\, \{V \subset \X\mid V\textup{ is open, }\eta_X(V)=0\}.\]  
$\Supp(\X)$ inherits the structure of a metric measure space from $\X$.  
\end{definition}

If $\X$ is separable, then we have that \[\eta_X(\X\setminus \Supp(\X))=0,\]
as one would want~\cite[\S 12.3]{aliprantis2006infinite}.  Hence, we assume separablility in the following definition.

\begin{definition}\label{Def:Measure_Iso}
An \emph{isomorphism} of separable metric measure spaces is a measure-preserving
isometry between their supports.  
\end{definition}

\begin{definition}
For $\eta$ a measure on a Polish space $( \AmbientSpace, \partial_ \AmbientSpace)$, we say $\eta$ has \emph{finite $p^{\text{th}}$ moment} if for any (hence all) $z_0\in  \AmbientSpace,$  
\[\int_{ \AmbientSpace} \partial_{ \AmbientSpace}(z,z_0)^p\, d\eta<\infty.\]
\end{definition}

We now record some standard facts about the metric properties of the distances we have defined above.

\begin{proposition}[Metric Properties of Distances on Measures]\label{Rem:Metric_Properties_Of_Distances_On_Measures}\mbox{}
\begin{enumerate}[(i)]
\item $\Pro$ is an extended metric on the set of all measures on a fixed metric space; that is, $\Pro$ can take infinite values, but otherwise satisfies all the properties of a metric.  Moreover, $\Pro$ restricts to a metric on finite measures\textup{~\cite[\S A.2.5]{daley2003}}.  
\item  The Gromov--Prohorov distance is a pseudometric on the class of Polish metric probability spaces, and $\GPr(\X,\Y)=0$ if and only if $\X$ and $\Y$ are isomorphic in the sense of \cref{Def:Measure_Iso}.  Thus, $\GPr$ descends to a metric on isomorphism classes of Polish metric probability spaces~\textup{\cite{greven2009convergence,janson2020gromov}}.
\item  For any $p\in [1,\infty)$, $\Wa^p$ is an extended metric on the set of all probability measures on a fixed Polish space, and it restricts to a metric on measures with finite $p^\mathrm{th}$ moment \textup{\cite[\S 6]{Villani2008}}.
\item For any $p\in [1,\infty)$, $\GW^p$ is a pseudometric on the class of Polish metric probability spaces with finite $p^{\mathrm{th}}$ moment, and $\GW^p(\X,\Y)=0$ if and only if $\X$ and $\Y$ are isomorphic.  Thus, $\GW^p$ descends to a metric on isomorphism classes of Polish metric probability spaces with finite $p^{\mathrm{th}}$ moment \textup{\cite{Sturm2006}}.
\end{enumerate}
\end{proposition}
The proof of \cref{Rem:Metric_Properties_Of_Distances_On_Measures}\,(iv) appears in \cite{Sturm2006} only for the case $p=2$, but the same proof works for all $p$.

\subsubsection{Comparison of the Prohorov and Wasserstein Distances}
%The following result relate the (Gromov--)Prohorov and (Gromov--)Wasserstein distances, thereby yielding
%(Gromov--)Wasserstein versions  of the main stability results in this paper (\cref{Cor:Degree_Rips_Stab} and \cref{Thm:Degree_Rips_Stab}).
\begin{proposition}[\cite{bjerkevik2020presentation}]\label{Prop:Prohorov_Wasserstein_Bounds}\mbox{}
\begin{enumerate}[(i)]
\item For any probability measures $\mu$ and $\eta$ on a common Polish space,
\[\Pro(\mu,\eta)\leq \min\left(\Wa^p(\mu,\eta)^{\frac{1}{2}},\Wa^p(\mu,\eta)^{\frac{p}{p+1}}\right).\]
\item For any Polish metric probability spaces $\mu$ and $\eta$,
\[\GPr(\mu,\eta)\leq \min\left(\GW^p(\mu,\eta)^{\frac{1}{2}},\GW^p(\mu,\eta)^{\frac{p}{p+1}}\right).\]
\end{enumerate}
\end{proposition}

\cref{Prop:Prohorov_Wasserstein_Bounds}\,(ii) follows immediately from \cref{Prop:Prohorov_Wasserstein_Bounds}\,(i).    To briefly explain how the bound of (i) arises, the bound $\Pro(\mu,\eta)\leq \Wa^1(\mu,\eta)^{\frac{1}{2}}$ appears in \cite{Su2002}.  Moreover, it is a standard fact that $\Wa^p\leq \Wa^q$ whenever $p\leq q$, so in fact, \[\Pro(\mu,\eta)\leq \Wa^p(\mu,\eta)^{\frac{1}{2}}\] for all $p\in [1,\infty)$.  In addition, a direct argument appearing in \cite{bjerkevik2020presentation} shows that \[\Pro(\mu,\eta)\leq \Wa^p(\mu,\eta)^{\frac{p}{p+1}}\]  for all $p\in [1,\infty)$.

\subsubsection{Metrization of Weak Convergence}
It is well known that both the Prohorov and Wasserstein distances metrize weak convergence of measures under suitable conditions.  We now briefly discus this.

\begin{definition}
Let $\mu$ be a measure on a topological space $T$.  A sequence of measures $\mu_1,\mu_2,\ldots$ on $T$ is said to \emph{weakly converge} to $\mu$ if \[\lim_{n\to \infty} \int f\, d\mu_n= \int f\, d\mu\] for all bounded, continuous functions $f:T\to \R$. 
\end{definition}

\begin{proposition}[Metrization of Weak Convergence {\cite[\S A.2.5]{daley2003}},\ {\cite[\S 6]{Villani2008}}]\label{Prop:Metrization_Of_Weak_Convergence}{}\mbox{}
\begin{enumerate}[(i)]
\item The Prohorov distance on a separable metric space metrizes weak convergence of finite measures.  
\item For any $p$, the $p$-Wasserstein distance on a bounded Polish space metrizes weak convergence of probability measures.  
\end{enumerate}
\end{proposition}

We will use \cref{Prop:Metrization_Of_Weak_Convergence}\,(i)  in \cref{Sec:Convergence_of_SC}, to study the convergence of Subdivision-\v Cech bifiltrations of random samples of a metric probability space.  We will not use \cref{Prop:Metrization_Of_Weak_Convergence}\,(ii) in this paper. 

\begin{remark}
The statement of \cref{Prop:Metrization_Of_Weak_Convergence}\,(ii) can in fact be generalized to one which holds for arbitrary Polish spaces.  For this, one has to place some restrictions on the probability measures; see \cite[\S 6]{Villani2008}.
\end{remark}

\subsection{Interleavings and Homotopy
  Interleavings}\label{Sec:Interleavings}

We now turn to the task of defining metrics on the space of bifiltrations, using the
formalism of {\em interleavings}.  Interleavings are ubiquitous in the
persistent homology literature, and their theory is well developed; e.g., see~\cite{chazal2012structure,bauer2015induced,bjerkevik2016stability,de2016categorified,bubenik2015metrics}.
Interleavings were first defined for $\R$-indexed diagrams
in~\cite{chazal2009proximity} (though the definition is already
implicit in the earlier work~\cite{cohen2007stability}) and for
$\R^n$-indexed diagrams with $n\geq 1$ in~\cite{lesnick2015theory}.
The most fundamental result about interleavings, called the
\emph{isometry theorem} or the \emph{algebraic stability
theorem}, is a generalization of the original stability theorems for
persistent homology.  It says that the interleaving distance
between pointwise finite-dimensional functors $X,Y\colon \R\to \Vect$ is equal to the bottleneck
distance between their barcodes 
\cite{chazal2012structure,lesnick2015theory,bauer2015induced,bjerkevik2016stability}.  

While we do not directly use the isometry theorem in this paper, it motivates the use of interleavings to formulate stability results for 2-parameter persistence.  As further motivation,  interleaving distances on multi-parameter filtrations and persistence modules satisfy universal properties, which indicate that these are principled choices of distances \cite{lesnick2015theory,blumberg2017universality}.

\subsubsection{Interleavings}\label{Sec:Strict_Interleavings}
Recall that in \cref{sec:bifilt}, we have defined $\OurPoset$ to be
the poset $(0,\infty)^{\op}\times (0,\infty)$.  We now define
interleavings between $\OurPoset$-indexed diagrams, adapting a
definition introduced in \cite{lesnick2012multidimensional}.  We give
a general form of the definition, where the shifts are not necessarily
translations; such generalizations have previously been considered in several
places~\cite{lesnick2012multidimensional,bubenik2015metrics,harker2019comparison}.

Throughout, we assume that $\R^{\op}\times \R$ is endowed with the
product partial order, i.e., 
\[
(a,b)\leq (c,d) \quad\textrm{if and only if}\quad a\geq c\enskip
\textrm{and} \enskip b\leq d.
\]
Recall that a morphism of posets $f\colon P\to Q$ is a
function such that $f(x)\leq f(y)$ whenever $x\leq y$.    

\begin{definition}
We define a \emph{forward shift} to be an automorphism of posets
\[
\gamma\colon \R^{\op}\times \R\to \R^{\op}\times \R
\]
such that $x\leq \gamma(x)$ whenever $x \in \OurPoset$.   
\end{definition}

\begin{example}\label{Ex:Shift}
In the context of interleavings, the standard
example of a forward shift is the translation $\tau^\delta$, given by
\[
\tau^\delta(x,y)=(x-\delta,y+\delta),
\]
where $\delta\geq 0$. 
\end{example}

\begin{example}\label{ex:forwardflow}
For any $c\geq 1$ and $\delta\geq 0$, the map
\[
(x,y)\mapsto (x-\delta,cy+\delta)
\]
is also a forward shift.  We will work with such forward shifts when
we study the stability of the degree bifiltrations. 
\end{example}

We now check that the composition of forward shifts is itself a forward shift.  The following lemma will be helpful.  
\begin{lemma}\label{Lem:Forward_Shifts}
If $\gamma$ is a forward shift and $(x,y)\in (-\infty,0]^{\op}\times (0,\infty)$, then $\gamma(x,y)\geq (0,y).$
\end{lemma}

\begin{proof}
For $x'\in (0,\infty)^{\op}$, we have that $(x',y)\in \OurPoset$, so $\gamma(x',y)\geq (x',y)$.  Since $(x,y)\geq (x',y)$ and $\gamma$ is a morphism of posets, we thus have that $\gamma(x,y)\geq (x',y)$.  Because this holds for all $x'\in (0,\infty)^{\op}$, the result follows.
\end{proof}

\begin{proposition}
The composition of two forward shifts is a forward shift.
\end{proposition}

\begin{proof}
Let $\gamma^1$ and $\gamma^2$ be forward shifts, and consider $x\in \OurPoset$.  We need to show that $x\leq \gamma^2\circ \gamma^1(x)$.  Note that $x\leq \gamma^1(x)$. 
If $\gamma^1(x)\in \OurPoset$, then we also have $\gamma^1(x)\leq \gamma^2\circ \gamma^1(x)$, and the result follows by transitivity.  If $\gamma^1(x)\not\in \OurPoset$, then the result follows from \cref{Lem:Forward_Shifts}.
\end{proof}

Recall that a category $\C$ is called \emph{thin} if for every $a,b\in \ob
\C$, there is at most one morphism in $\C$ from $a$ to $b$.  
If $\C$ is thin, $F\colon \C\to \D$ is any functor, and $g\colon a\to b$ is a morphism in $\C$, we denote $F(g)$ as $F_{a,b}$.

\begin{definition}
For $\gamma$ and $\kappa$ two forward shifts, let the
\emph{$(\gamma,\kappa)$-interleaving category}, denoted
$\I^{\gamma,\kappa}$, be the thin category with object set $\OurPoset
\times \{0,1\}$ and a morphism $(r,i)\to (s,j)$ if and only if one of
the following is true:
\begin{itemize}
\item $i=j$ and $r\leq s$,
\item $i=0$, $j=1$, and $\gamma(r)\leq s$,
\item $i=1$, $j=0$, and $\kappa(r)\leq s$.
\end{itemize}
\end{definition}

We have embeddings 
\[
E^0,E^1\colon \OurPoset\to \I^{\gamma,\kappa}.
\]
mapping $r\in \OurPoset$ to $(r,0)$ and $(r,1)$, respectively.

\begin{definition}\label{Def:Interleaving}
\mbox{}
\begin{enumerate}[(i)]
\item Given forward shifts $\gamma,\kappa$ and any category $\C$, we define a
\emph{$(\gamma,\kappa)$-interleaving} between functors $\F,\G\colon
\OurPoset \to \C$ to be a functor 
\[
\Hcal\colon \I^{\gamma,\kappa}\to \C
\]
such that $\Hcal\circ E^0=\F$ and $\Hcal\circ E^1=\G$.  
\item  We refer to a
$(\gamma,\gamma)$-interleaving simply as a $\gamma$-interleaving, and for $\delta\geq 0$, we refer to a $\tau^\delta$-interleaving as a $\delta$-interleaving.
\end{enumerate}  
\end{definition}

In this work, we consider the cases $\C=\Top$
and $\C=\Vect$.

We can now define an extended pseudometric on bifiltrations in terms of interleavings.

\begin{definition}\label{Def:Interleaving_Distance}
We define 
\[
d_I\colon \ob \C^{\OurPoset} \times \ob \C^{\OurPoset} \to [0,\infty],
\]
the \emph{interleaving distance}, by taking  
\[
d_I(\F,\G)=\inf\, \{\delta \mid \F\textup{ and  }\G\textup{
    are }\delta\textup{-interleaved}\}.
\]
\end{definition}
It is easily checked that $d_I$ is an extended pseudometric.

\subsubsection{Homotopy Interleavings}\label{Sec:HI}

In prior work~\cite{blumberg2017universality}, we introduced
\emph{homotopy interleavings}, homotopical generalizations of
interleavings which are useful for formulating TDA results directly
at the space level, rather than on the algebraic (homological) level.
In~\cite{blumberg2017universality}, we defined $\delta$-homotopy interleavings for $\R$-indexed diagrams
of topological spaces.  The definition generalizes without difficulty
to our setting, as we now explain.

Recall that a {\em diagram of spaces} is a functor $\F \colon \C \to
\Top$, where $\C$ is 
a small category.  
The following definition is the starting point for the
homotopy theory of diagrams; see, e.g., ~\cite{hirschhorn}.

\begin{definition}\label{Def:Weakly_Equivalent}
For $\C$ a small category and functors $\F,\G\colon \C\to \Top$, a
natural transformation $f\colon \F\to \G$ is an \emph{objectwise weak
  equivalence} if $f_x$ is a weak homotopy equivalence for all $x\in
\C$. We say that diagrams $\F$ and $\G$ are \emph{weakly equivalent},
and write $\F\htp \G$ if there is a finite sequence of functors
\[\F=\Hcal_1,\Hcal_2,\ldots, \Hcal_n=\G,\]
with each $K_i$ also a functor from $\C$ to $\Top$, such that for all $i\in \{1,.\dots,n-1\}$
there exists either a weak equivalence $\Hcal_i\to  \Hcal_{i+1}$ or a weak
equivalence $\Hcal_{i+1}\to \Hcal_i$. 
\end{definition}

Using the notion of weak equivalence, we define a homotopical
refinement of interleavings:

\begin{definition}\label{def:HI}\mbox{}
\begin{enumerate}[(i)]
\item Given forward shifts $\gamma$ and $\kappa$, functors $\F,\G\colon
\OurPoset\to \Top$ are \emph{$(\gamma,\kappa)$-homotopy interleaved}
if there exist functors $\F',\G'\colon \OurPoset\to \Top$ with $\F'\htp \F$ and
$\G'\htp \G$ such that $\F'$ and $\G'$ are $(\gamma,\kappa)$-\textup{interleaved}. 
\item  In analogy with \cref{Def:Interleaving}\,(ii), we say that $\F$ and $\G$ are $\gamma$-homotopy interleaved if they are $(\gamma,\gamma)$-interleaved, and we say that they are $\delta$-homotopy interleaved if they are $\tau^\delta$-interleaved.
\end{enumerate}
\end{definition}

\begin{definition}
The \emph{homotopy interleaving distance} between functors $\F,\G\colon \OurPoset\to \Top$ is given by 
\[
d_{HI}(\F,\G):=\inf\, \{\delta \mid \F,\G\textup{ are
}\delta\textup{-homotopy interleaved}\}.
\] 
\end{definition}

The following proposition implies in particular that $d_{HI}$ is an extended pseudometric.

\begin{proposition}[Generalized Triangle Inequality for Homotopy Interleavings]\label{Prop:Triangle_Inequality}
Consider functors $\F,\G,\Hcal\colon \OurPoset\to \Top$ and forward shifts $\gamma^1,\gamma^2,\kappa^1,\kappa^2$.  If  $\F,\G$ are $(\gamma^1,\kappa^1)$-homotopy interleaved and $\G,\Hcal$ are $(\gamma^2,\kappa^2)$-homotopy interleaved then $\F,\Hcal$ are $(\gamma^2\circ \gamma^1,\kappa^1\circ \kappa^2)$-homotopy interleaved.
\end{proposition}

The analogue of \cref{Prop:Triangle_Inequality} is established
in~\cite{blumberg2017universality} for $\delta$-interleavings of
$\R$-indexed diagrams of spaces, using a homotopy Kan extension.  The
proof of \cref{Prop:Triangle_Inequality} is essentially the same; we omit it.

For forward shifts $\gamma^1$ and $\gamma^2$, write $\gamma^1\leq \gamma^2$ if $\gamma^1(x)\leq \gamma^2(x)$ for all $x\in \OurPoset$.  

\begin{proposition}[Cf. {\cite[Proposition 2.2.12]{bubenik2015metrics}}]\label{Prop:Interleaving_Monotonicity}
If $\gamma^1\leq \gamma^2$ and $\kappa^1\leq\kappa^2$ are forward shifts and $\F,\G\colon \OurPoset\to \Top$ are $(\gamma^1,\kappa^1)$-homotopy interleaved, then $\F,\G$ are $(\gamma^2,\kappa^2)$-homotopy interleaved.
\end{proposition}

\begin{proof}
A $(\gamma^1,\kappa^1)$-interleaving restricts to a $(\gamma^2,\kappa^2)$-interleaving.
\end{proof}

\begin{remark}\label{Rem:Strict_Variants}
The obvious analogues of \cref{Prop:Triangle_Inequality} and \cref{Prop:Interleaving_Monotonicity} for strict interleavings also hold.
\end{remark}

The following result, whose easy proof we omit, tells us that homology preserves interleavings in the expected way:

\begin{proposition}\label{Proposition:Homotopy_to_Homology_Interleavings}
If $\F,\G\colon \OurPoset \to \Top$ are homotopy $(\gamma,\kappa)$-interleaved, then $H_i \F,H_i\G$ are $(\gamma,\kappa)$-interleaved for all $i\geq 0$.
\end{proposition}

\section{Main Results}\label{sec:theorems}
\subsection{Stability of the Multicover and Subdivision Bifiltrations}\label{Sec:_Sbdivision_Stability}

We now prove \cref{Thm:Multicover_and_Subdivision_Cech_Stability}, our
stability result for the multicover and 
subdivision bifiltrations. Let us recall the statement of the theorem.  First, we remind the reader of \cref{Not:Counting_Measures}: For $X$ a non-empty, finite metric space, $\mu_X$ denotes the uniform probability measure on $X$, and if $X$ is a subset of a metric space $\AmbientSpace$, then $\nu_X$ denotes the pushforward of  $\mu_X$ into $\AmbientSpace$.

\begin{reptheorem}{Thm:Multicover_and_Subdivision_Cech_Stability}
\mbox{}
\begin{enumerate}
\item[(i)] For $X$ and $Y$ non-empty, finite subsets of a metric
  space $\AmbientSpace$, \[d_{I}(\nMult(X),\nMult(Y)) \leq \Pro(\nu_X,\nu_Y).\]
\item[(ii)] In (i), if we assume further that $\AmbientSpace$ is good, then also \[d_{HI}(\nSCech(X),\nSCech(Y)) \leq \Pro(\nu_X,\nu_Y).\]
\item[(iii)] For any non-empty, finite metric spaces $X$ and $Y$, 
\[
d_{HI}(\nSRips(X),\nSRips(Y)) \leq \GPr(\mu_X,\mu_Y).
\] 
\end{enumerate}
\end{reptheorem}

We first prove
\cref{Thm:Multicover_and_Subdivision_Cech_Stability}\,(i), our
stability result for the multicover filtration.  In fact, we prove the following generalization: 

\begin{theorem}[Stability of Measure Bifiltrations]\label{Thm:Measure_Stability}
For any measures $\mu$ and $\eta$ on a common metric space $\AmbientSpace$, 
\[d_{I}(\MB(\mu),\MB(\eta)) \leq \Pro(\mu,\eta).\]
\end{theorem}

\cref{Thm:Multicover_and_Subdivision_Cech_Stability}\,(i) follows immediately from \cref{Thm:Measure_Stability} by taking $\mu=\nu_X$ and $\eta=\nu_Y$, 
\begin{remark}\label{Rem:Unnormalized_Results}

In our notation, $\nu_X$ and $\nu_Y$ are the normalized 
counting measures of $X$ and $Y$, respectively.  Considering instead the \emph{unnormalized} counting measures,
\cref{Thm:Measure_Stability} also yields an unnormalized version of
\cref{Thm:Multicover_and_Subdivision_Cech_Stability}\,(i), which
requires no finiteness assumption on $X$ and $Y$.  In fact,
all of our stability results also admit analogous unnormalized variants.  We
have chosen to emphasize the normalized version of the results in our
exposition because the definition of the Wasserstein distances does
not make sense for pairs of measures with different total measure, so
only the normalized versions of our results imply corresponding
stability bounds for the (Gromov--)Wasserstein distances.
\end{remark}

\begin{proof}[Proof of \cref{Thm:Measure_Stability}]
It suffices to show
that that for any $\delta>\Pro(\mu,\eta)$, $\MB(\mu)$ and
$\MB(\eta)$ are $\delta$-interleaved.  As above, for $x\in \AmbientSpace$ and $r> 0$, let
$\Ball(x,r)$ denote the open ball of radius $r$ centered at $x$, and let
$\bar \Ball(x,r)$ denote its closure.  For any $(k,r)\in \OurPoset$, if
$x\in \MB(\mu)_{(k,r)}$, then $\mu(B(x,r))\geq k$.  As $\Ball(x,r)\subset \bar \Ball(x,r)$, we have that $\mu(\bar
\Ball(x,r))\geq k$.  By the definition of $\Pro$, 
\[\eta(\bar
\Ball(x,r)^\delta)+\delta \geq \mu(\bar \Ball(x,r)),\] 
so $\eta(\bar \Ball(x,r)^\delta)\geq k-\delta$.  But by the triangle inequality
\[
\bar \Ball(x,r)^\delta\subset \Ball(x,r+\delta),
\]
so $\eta(\Ball(x,r+\delta))\geq k-\delta$.  Thus,
if $k>\delta$, so that $\MB(\eta)_{(k-\delta,r+\delta)}$ is
defined, then $x\in \MB(\eta)_{(k-\delta,r+\delta)}$; it follows that 
\[
\MB(\mu)_{(k,r)}\subset \MB(\eta)_{(k-\delta,r+\delta)}.
\]
The same is true with the roles of $\mu$ and $\eta$ reversed.  Thus,
$\MB(\mu)$ and $\MB(\eta)$ are $\delta$-interleaved, with the
interleaving given by inclusion maps.
\end{proof}

\cref{Thm:Multicover_and_Subdivision_Cech_Stability}\,(ii) now follows
immediately from statement (ii) of the following version of
Sheehy's multicover nerve theorem.

\begin{theorem}[Multicover Nerve Theorem for
Bifiltrations]\label{Thm:Multicover}\mbox{}
Given a good metric space $\AmbientSpace$ and $X \subset
\AmbientSpace$,
\begin{enumerate}[(i)]
\item the unnormalized bifiltrations $\SCech(X)$ and $\Mult(X)$ are weakly equivalent,
\item if $X$ is finite, then the normalized bifiltrations $\nSCech(X)$ and $\nMult(X)$ are also weakly equivalent. 
\end{enumerate}
\end{theorem}

We prove \cref{Thm:Multicover}\,(i) in~\cref{Sec:Multicover_Nerve}, following~\cite{cavanna17when}.  For $X$ finite, $\nSCech(X)$ and $\nMult(X)$ are obtained from $\SCech(X)$ and $\Mult(X)$ by the same reparameterization, so \cref{Thm:Multicover}\,(ii) follows immediately
from (i).

\begin{proof}[Proof of \cref{Thm:Multicover_and_Subdivision_Cech_Stability}\,(iii)]
Our proof is an adaptation of the original proof of the Rips stability
theorem \cite{chazal2009gromov}.  For any
$\delta>\GPr(\mu_X,\mu_Y)$, there exists a finite metric space $Z$ and
a pair of isometric embeddings $\varphi\colon X\to Z$, $\psi \colon Y\to Z$
such that $\Pro(\varphi_*(\mu_X), \psi_*(\mu_Y))<\delta$.  Let $\R^{Z}$ denote the metric space of functions $Z\to \R$, with the
sup norm metric, and let $K\colon Z\to \R^Z$ be the \emph{Kuratowski embedding}, defined by $K(z)(y)=d_Z(y,z)$.  $K$ is an isometric embedding.  We write \[X'=K\circ \varphi(X)\subset \R^Z\qquad\textup{and}\qquad Y'=K\circ \psi(Y)\subset \R^Z.\]  The Rips and \v Cech complexes of sets embedded in $\R^{Z}$ are identical, so we have 
\[
\nSRips(X)=\nSCech(X') \qquad \textrm{and} \qquad \nSRips(Y)=\nSCech(Y').
\]
Moreover, we have that
\[\Pro(\nu_{X'},\nu_{Y'})=\Pro(\varphi_*(\mu_X), \psi_*(\mu_Y))<\delta.\]
Hence, by \cref{Thm:Multicover_and_Subdivision_Cech_Stability}\,(ii),
$d_{HI}(\nSCech(X'),\nSCech(Y'))<\delta$.  The result follows. 
\end{proof}

\subsection{Stability of the Degree Bifiltrations}\label{Sec:Stability_of_Degree_Bifiltrations}

Next, we prove \cref{Thm:Degree_Rips_Stab}, our stability result for
the degree bifiltrations, using the results of the
previous section.  To start, we apply the multicover nerve theorem
(\cref{Thm:Multicover_and_Subdivision_Cech_Stability}\,(ii)) to show
that the subdivision-\v Cech and degree-\v Cech bifiltrations are
homotopy interleaved.

Recalling Example~\ref{ex:forwardflow}, for $\delta\geq 0$ we define
the forward shift 
\[
\begin{aligned}
\gamma^\delta \colon \R^{\op} \times \R \to \R^{\op} \times \R, \\
\gamma^\delta(k,r)=(k-\delta,3r+\delta).
\end{aligned}
\]

Let us now recall the statement of \cref{Thm:Degree_Rips_Stab}:
\begin{reptheorem}{Thm:Degree_Rips_Stab}
    \mbox{}
\begin{enumerate}
\item[(i)] If $X$ and $Y$ are non-empty finite subsets of a good metric space, then
  $\nDCech(X)$ and $\nDCech(Y)$ are $\gamma^\delta$-homotopy
  interleaved for all $\delta>\Pro(\nu_X,\nu_Y),$
\item[(ii)] If $X$ and $Y$ are non-empty finite metric spaces, then  $\nDRips(X)$ and $\nDRips(Y)$ are
  $\gamma^\delta$-homotopy interleaved for all $\delta>\GPr(\mu_X,\mu_Y).$
\end{enumerate}    
\end{reptheorem}

In what follows, we let $\Id$ denote the identity function on
$\R^{\op}\times \R$, and we write $\gamma^0$ simply as $\gamma$.

\begin{proposition}\label{Prop:Degree_Cech_vs_Subdivision_Cech}
For $X$ a non-empty, finite subset of a good metric space, $\nDCech(X)$ and $\nSCech(X)$ are $(\gamma,\Id)$-homotopy interleaved.
\end{proposition}

\begin{proof}
Our proof uses the \emph{persistent nerve theorem} \cite{chazal2008towards}, a standard variant of the nerve theorem which we discuss in the following section; see \cref{Thm:Persistent_Nerve}.  Let $(\AmbientSpace,\partial_{\AmbientSpace})$ denote the ambient good metric space, and consider the
bifiltration 
\[
\begin{aligned}
&\DCov(X) \colon \OurPoset\to \Top, \\
&\DCov(X)_{(k,r)} =\{y\in \AmbientSpace\mid d(y,x)<r\textup{ for some $x\in X$ a
  vertex of $\nDCech_{(k,r)}$}\}.
\end{aligned}
\]
By the persistent nerve theorem, $\DCov(X)$ is weakly equivalent to
$\nDCech(X)$.   Thus, by \cref{Thm:Multicover}\,(ii) it suffices to show that
$\DCov(X)$ and $\nMult(X)$ are $(\gamma,\Id)$-interleaved.

If $x\in \DCov(X)_{(k,r)}$, then there exists a point $p\in X$ with
$\partial_{\AmbientSpace}(x,p)<r$, and a subset $W\subset X$ of size at least $k|X|$ such that
 $\partial_{\AmbientSpace}(p,w)<2r$ for all $w\in W$.  By the triangle inequality,
$\partial_{\AmbientSpace}(x,w)<3r$ for all $w\in W$, so $x\in \nMult(X)_{(k,3r)}$.  This shows
that 
\[
\DCov(X)_{(k,r)}\subset \nMult(X)_{(k,3r)}.
\]  
Conversely, if $x\in \nMult(X)_{(k,r)}$ then there is a subset
$W\subset X$ of size at least $k|X|$ such that $\partial_{\AmbientSpace}(x,w)<r$ for all $w\in W$.  Then, for any $w,w'\in W$,
$\partial_{\AmbientSpace}(w,w')<2r$ by the triangle inequality, so each element of $W$ is a vertex of
$\nDCech_{(k,r)}$. Thus, $x\in \DCov_{(k,r)}$.  This shows that \
\[
\nMult(X)_{(k,r)}\subset \DCov(X)_{(k,r)}.
\]
It follows that that $\DCov(X)$ and $\nMult(X)$ are indeed
$(\gamma,\Id)$-interleaved, with the interleaving maps the
inclusions.
\end{proof}

\begin{corollary}\label{Cor:DRips_And_SRips}
For any non-empty, finite metric space $X$, $\nDRips(X)$ and $\nSRips(X)$ are
$(\gamma,\Id)$-homotopy interleaved.
\end{corollary}

\begin{proof}
Without loss of generality, we may regard  
$X$ as a subset of $\R^X$ via the Kuratowski embedding; see the proof of \cref{Thm:Multicover_and_Subdivision_Cech_Stability}\,(iii).  For point
sets in $\R^X$, Rips and \v Cech complexes are equal, so 
\[
\nDRips(X)=\nDCech(X) \qquad \textrm{and} \qquad \nSRips(X)=\nSCech(X).
\]
The result now follows
from~\cref{Prop:Degree_Cech_vs_Subdivision_Cech}.
\end{proof}

\begin{proof}[Proof of \cref{Thm:Degree_Rips_Stab}]
Item (i) follows immediately from \cref{Thm:Multicover_and_Subdivision_Cech_Stability}\,(ii), \cref{Prop:Degree_Cech_vs_Subdivision_Cech},  and the ``generalized triangle inequality for homotopy interleavings''  (\cref{Prop:Triangle_Inequality}).  Similarly, (ii) follows from \cref{Thm:Multicover_and_Subdivision_Cech_Stability}\,(iii), \cref{Cor:DRips_And_SRips}, and \cref{Prop:Triangle_Inequality}.
\end{proof}

\begin{remark}
One can also prove \cref{Thm:Degree_Rips_Stab}\,(i) without considering the relationship between the degree and subdivision bifiltrations.  This approach avoids use of both the multicover nerve theorem and the generalized triangle inequality for homotopy interleavings; only the usual persistent nerve theorem (\cref{Thm:Persistent_Nerve}) and a generalized triangle inequality for strict interleavings are needed.  \cref{Thm:Degree_Rips_Stab}\,(ii) also can be proven this way.
\end{remark}
The following proposition tells us that the constant 3 in the definition of $\gamma^\delta$ is tight.  For $c\in [1,\infty)$, define the forward shift $\gamma^{\delta,c}$ by 
\[\gamma^{\delta,c}(k,r)=(k-\delta,cr+\delta),\]
and note that $\gamma^\delta=\gamma^{\delta,3}$.

\begin{proposition}\label{Prop:Degree_Tightness}
For any $c\in [1,3)$, neither statement (i) nor (ii) of \cref{Thm:Degree_Rips_Stab} is true if we replace $\gamma^{\delta}$ in the statement with $\gamma^{\delta,c}$.
\end{proposition}

We prove this proposition in Appendix~\ref{sec:tightness}.  Moreover, we have the following:

\begin{remark}[Discontinuity of the Degree Bifiltrations]\label{Rmk:Discontinuity}
The map sending a finite metric space to its normalized degree-Rips bifiltration is discontinuous with respect to the Gromov--Prohorov distance on uniform measures and the homotopy interleaving distance.  To see this, let $Y=\{-1,1\}\subset \R$.  For $n\in \mathbb N$, let $Z^n\subset \R$ be any set consisting of $n$ points in the interval $[-1-\frac{1}{n},-1+\frac{1}{n}]$ and $n$ points in the interval $[1-\frac{1}{n},1+\frac{1}{n}]$, and let $Y^n=Z^n\cup \{0\}$.  It is easy to check that $\mu_{Y^n}$ converges to $\mu_{Y}$ in the Gromov--Prohorov metric as $n\to \infty$.  However, a simple calculation similar to one appearing in the proof of \cref{Prop:Degree_Tightness} shows that $H_0(\nDRips(Y^n))$ does not converge to $H_0(\nDRips(Y))$ in the interleaving distance.  Hence, $\nDRips(Y^n)$ does not converge to $\nDRips(Y)$ in the homotopy interleaving distance, which establishes the claimed discontinuity.  

The same example also gives that the map sending a finite subspace of a good metric space to its normalized degree-\v Cech bifiltration is discontinuous with respect to the Prohorov distance on normalized measures and the homotopy interleaving distance.
For this, it suffices to note that $\nu_{Y^n}$ also converges to $\nu_{Y}$ in the Prohorov metric as $n\to \infty$, and that the degree-Rips and degree-\v Cech bifiltrations have the same 1-skeleta, hence isomorphic homology in degree 0.
\end{remark}

Harker et al.\ have introduced a variant of the standard Rips stability theorem (\cref{Thm:Classical_Stability}\,(ii)) for the case where one metric space is a subset of the other \cite[Proposition 5.6]{harker2019comparison}.  We now give an analogous variant of \cref{Thm:Degree_Rips_Stab}, which we will use in \cref{Sec:RIVET_Point_Clouds}.  \cref{Thm:Multicover_and_Subdivision_Cech_Stability} also admits an analogous variant, which we do not write out.

For $C\in (0,1]$, define the forward shift $\kappa^{C}$ by $\kappa^C(k,r)=(Ck,r)$.
\begin{proposition}\label{Prop:Degree_Subspaces} 
\mbox{}
\begin{enumerate}
\item[(i)] If $X\subset Y$ are non-empty, finite subsets of a good metric space, then
  $\nDCech(X)$ and $\nDCech(Y)$ are  $(\kappa^{|X|/|Y|},\gamma^\delta)$-homotopy
  interleaved for all $\delta>\Pro(\nu_X,\nu_Y)$.
\item[(ii)] If $X\subset Y$ are non-empty, finite metric spaces, then $\nDRips(X)$ and $\nDRips(Y)$ are $(\kappa^{|X|/|Y|},\gamma^\delta)$-homotopy interleaved for all $\delta>\Pro(\mu_X,\mu_Y)$.
\end{enumerate}
\end{proposition}
Evidently, the hypotheses of \cref{Prop:Degree_Subspaces} are stronger than those of \cref{Thm:Degree_Rips_Stab}, and using \cref{Prop:Interleaving_Monotonicity},  it can be checked that the conclusions are stronger as well.

\cref{Prop:Degree_Subspaces} is proven in essentially the same way as \cref{Thm:Degree_Rips_Stab}, using the inclusion $X\hookrightarrow Y$ in a straightforward way to strengthen the result; we omit the details.

\subsection{Weak Law of Large Numbers for Subdivision-\v Cech Bifiltrations}\label{Sec:Convergence_of_SC}
 Let $\X=(\X,\partial_\X,\eta_\X)$ be a separable metric probability
 space (\cref{Def:MM_Space}).  To conclude this section, we apply
 \cref{Thm:Measure_Stability} to show that the subdivision-\v Cech
 bifiltration of an i.i.d. sample of $\X$ 
 converges almost surely in the homotopy
 interleaving distance to the multicover bifiltration of $\X$, as the
 sample size tends to $\infty$.

We begin with some notation: Let $x_1,x_2,\ldots$ be a sequence of independent random variables taking values in $\X$, each having law $\eta_\X$; for $m\geq 1$, let $X_m=\{x_1,\ldots,x_m\}$.   Let $\eta_m$ denote the empirical distribution of $X_m$; that is, using \cref{Not:Counting_Measures}, $\eta_m=\nu_{X_m}.$

We will use the following standard theorem about convergence of empirical measures.

\begin{theorem}[{\cite[Theorem 11.4.1]{dudley2018real}}]\label{Thm:Convergence_of_Emperical_Measures}
Almost surely, the empirical distributions $\eta_m$ weakly converge to $\eta_\X$ as $m\to \infty$.  
\end{theorem}

Here is our convergence result.

\begin{theorem}\label{Thm:Cech_Convergence}
The random variables $\nSCech(X_m)$ converge almost surely to $\MB(\eta_X)$ in the homotopy interleaving distance as $m\to \infty$.
\end{theorem}

\begin{proof}[Proof of \cref{Thm:Cech_Convergence}]
By \cref{Thm:Convergence_of_Emperical_Measures} and \cref{Prop:Metrization_Of_Weak_Convergence}\,(i), $\eta_m$ converges to $\eta_\X$ in the Prohorov distance.  Thus, by \cref{Thm:Measure_Stability}, $\nMult(X_m)= \MB(\eta_m)$ converges to $\MB(\eta_\X)$ in the interleaving distance.  Since $\nSCech(X_m)$ is weakly equivalent to $\nMult(X_m)$ by \cref{Thm:Multicover}\,(ii), this implies that $\nSCech(X_m)$ converges to $\MB(\eta_\X)$ in the homotopy interleaving distance.
\end{proof}

\section{The Multicover Nerve Theorem}\label{Sec:Multicover_Nerve}

We now turn to the proof of \cref{Thm:Multicover}\,(i), the multicover
nerve theorem for bifiltrations. The original version of the multicover nerve theorem \cite{sheehy2012multicover} concerns the 1-parameter persistence modules obtained by fixing the parameter $k$ of the multicover bifiltration while varying the parameter $r$.  A subsequent paper by Cavanna, Gardner, and Sheehy gives a different proof of the same result~\cite[Appendix B]{cavanna17when}.  The latter approach centers around the observation
that a proof of the multicover nerve theorem is essentially already
implicit in a standard proof of the nerve theorem, as given
e.g., in~\cite[\S 4.G]{hatcher2002algebraic}.  Just one additional step
is required, where one checks that a certain map of homotopy colimits
is a homotopy equivalence.  This step was in fact omitted
in~\cite{cavanna17when}, but is not too difficult to fill in.

While the proof in ~\cite{sheehy2012multicover} seems to not extend readily to a proof of the result for bifiltrations (\cref{Thm:Multicover}\,(i)), the proof of~\cite{cavanna17when} does extend, modulo the omitted step.  Here, we give a complete proof \cref{Thm:Multicover}\,(i), following \cite{cavanna17when} and filling in the gap.  In fact, we prove a generalization which holds for a wider class of filtered covers (\cref{Thm:General_Multicover_Nerve}), since this is not much more difficult.  Our argument for the omitted step draws on an idea appearing in Sheehy's original proof~\cite{sheehy2012multicover}, and also uses the
standard fact that homotopy final functors induce equivalences on
homotopy colimits~(\cref{Thm:Homotopy_Finality}).

\subsection{Nerves and Homotopy Colimits}

\begin{definition}[Nerves]\mbox{}\label{Def:Diagramatic_Nerve}
\begin{enumerate}[(i)]
\item Let $U$ be a cover of a topological space.  The set of
  finite subsets of $U$ with non-empty common intersection define a
  simplicial complex, called the \emph{nerve} of $U$ and denoted
  $N(U)$.
\item The \emph{nerve} of a poset $P$, denoted $N(P)$, is the simplicial
  complex whose $k$-simplices are chains $p_0<p_1<\cdots < p_k$ of
  elements in $P$, with the face relation given by removing elements
  in a chain. 
\end{enumerate}
\end{definition}

We record a few simple facts about nerves of posets that we will use in the
proof the multicover nerve theorem.

\begin{proposition}\label{Prop:Min_Elt_Contractible}\label{Prop:Subdivision}\mbox{}
\begin{enumerate}[(i)]
\item For $P$ the face poset of a simplicial complex $S$, we have
  $N(P)= \Bary(S)$.  
\item For any poset $P$, $N(P)$ and $N(P^{\op})$ are canonically isomorphic.
\item If a poset $P$ has a minimum element, then $N(P)$ is contractible. 
\end{enumerate}
\end{proposition}

\begin{proof}
Items (i) and (ii) are trivial.  To prove (iii), note that when $P$ has a minimum element $p$, $N(P)$ is the cone on
$N(P\setminus \{p\})$ and so is contractible. 
\end{proof}

We will make use of the classical Bousfield--Kan formula for the
homotopy colimit, which we now review.  We will only need to consider
the case where the index category is a poset $P$.  For $\sigma\in
N(P)$ a simplex, let $\sigma_0\in \sigma$ denote the minimum element.
Let $\Delta^k$ denote a $k$-simplex, regarded as a topological space.  The functor 
\[
\hocolim_P\colon \Top^P\to \Top
\]
is defined on objects by 
\[
\hocolim_P F = \left(\bigsqcup_{\sigma\in N(P)} \Delta^k\times
F_{\sigma_0} \right)/\sim,
\]
where for each $\sigma \in N(P)$ and facet of $\tau$ of $\sigma$, the
quotient relation glues $\Delta^k\times F_{\sigma_0}$ to
$\Delta^{k-1}\times F_{\tau_0}$ along $\tau\times F_{\sigma_0}$ via
the map $F_{\sigma_0,\tau_0}$.  The action of $\hocolim_P F$ on morphisms (i.e., natural
transformations) in $\Top^P$ is defined in the obvious way.  The
coordinate projections $\Delta^k\times F_{\sigma_0} \twoheadrightarrow
F_{\sigma_0}$ induce a comparison map $\hocolim_P F\to \colim_P F$, which is usually not a weak equivalence.

The homotopy colimit is homotopy invariant, in the sense of the
following standard result: 

\begin{proposition}[{\cite[\S 14.5]{riehl}, \cite[Theorem 8.3.7]{munson2015cubical}}]\label{Prop:Homotopy_Invariance_Of_Homotopy_Colimit}
Given functors 
\[
F,G \colon P \to \Top
\]
and an objectwise (weak) homotopy equivalence $f\colon F\to G$, 
the induced map 
\[
\hocolim_P(f)\colon \hocolim_P F\to \hocolim_P G
\]
is a (weak) homotopy equivalence.
\end{proposition}

\begin{remark}\label{Rem:Included_Maps_HCLs}
For functors $F\colon P\to Q$ and $G\colon Q\to \Top$, $F$ induces a
map of homotopy colimits 
\[
F_* \colon \hocolim_P G\circ F \to \hocolim_Q G.
\]
\end{remark}

Let $F\colon P\to Q$ be a functor of posets and $q\in Q$.  We define the category
$q \downarrow F$ to be the subposet of $P$ given by \[q\downarrow
F=\{p\in P\mid q\leq F(p)\}.\]

\begin{definition}
A functor $F\colon P\to Q$ of posets is said to be \emph{homotopy final} if
$N(q\downarrow F)$ is contractible for all $q\in Q$.
\end{definition}

The following result is a useful tool for computing homotopy colimits:

\begin{theorem}[{\cite[Theorem 8.5.6]{riehl}, \cite[Theorem 8.6.5]{munson2015cubical}}]\label{Thm:Homotopy_Finality}
If $F\colon P\to Q$ is a homotopy final functor of posets, then for
any functor 
$G\colon Q\to \Top$, 
\[
F_* \colon \hocolim_P G \circ F \to \hocolim_Q G
\]
is a homotopy equivalence.
\end{theorem}

\subsection{The Nerve Theorem}
We now outline a proof of a version of the nerve theorem for open covers, following \cite[\S 4.G]{hatcher2002algebraic}, \cite[\S 2]{dugger2004topological}, and \cite[\S 5]{bauer2022unified}.  %But whereas \cite[\S 4.G]{hatcher2002algebraic} considers only paracompact spaces, we give a result for arbitrary spaces by appealing to results in \cite[\S 2]{dugger2004topological}, as previously done in .
See \cite{bauer2022unified} for a thorough treatment of other variants of the nerve theorem, emphasizing the functorial formulations most useful in TDA.  

For $U$ a set of topological spaces, let $P^U$ denote the opposite poset of the face poset of
$N(U)$.  We have a functor 
\[\begin{aligned}
&D^U\colon P^U\to \Top,\\
&D^U([U_1,\ldots,U_k])=\bigcap_{i=1}^k U_{i},
\end{aligned}
\]
where the internal maps of $D^U$ are inclusions.

Recall that two topological spaces are said to be \emph{weakly homotopy equivalent}, or simply \emph{weakly equivalent}, if they are connected by a zigzag of weak homotopy equivalences.

\begin{definition} 
We say that a cover $U$ of a topological space is (weakly) \emph{good} if the
common intersection of any finite subset of $U$ is empty or (weakly) homotopy equivalent to a point.
\end{definition}

\begin{theorem}[Nerve theorem for open covers {\cite[\S 4.G]{hatcher2002algebraic}, \cite[Theorem 5.9]{bauer2022unified}}]\label{Thm:Weil_Nerve}~
Let $U$ be an open cover of a topological space $X$.  
\begin{itemize}
\item[(i)] If $U$ is weakly good, then $X$ is weakly equivalent to $N(U)$.  
\item[(ii)] If $U$ is good and $X$ is paracompact, then $X$ is homotopy equivalent to $N(U)$.  
\end{itemize}
\end{theorem}

\begin{proof}[Outline of proof]
Let $\ast \colon P^U\to \Top$ denote the constant functor with value a
point, and let $Z=\hocolim_{P^U} (D^U)$.  Consider the natural transformation $D^U\to \ast$.  
If $U$ is (weakly) good, then by \cref{Prop:Homotopy_Invariance_Of_Homotopy_Colimit}, the induced map 
\[
Z\to \hocolim_{P^U} \ast
\]
is a (weak) homotopy equivalence.  But
\begin{equation}\label{Eq:Nerve_As_Hocolim}
\hocolim_{P^U} \ast \cong N(P^U)\cong \Bary(N(U)) \cong N(U),
\end{equation}
where we get the middle homeomorphism from \cref{Prop:Subdivision}\,(i) and (ii).  Thus,
we obtain a (weak) homotopy equivalence 
\[
\rho^1\colon Z\to N(U).
\]
Let 
\[
\rho^2\colon Z\to \colim_{P^U} D^U=X
\]
denote the natural map from the homotopy colimit to the colimit.  If $X$ is paracompact, then a
partition of unity argument constructs a homotopy inverse for
the map $\rho^2$ in the proof of \cref{Thm:Weil_Nerve}, so $\rho^2$ is a homotopy equivalence \cite[\S 4.G]{hatcher2002algebraic}.   
For $X$ not necessarily paracompact, it follows from the results of \cite[\S 2]{dugger2004topological} that $\rho^2$ is a weak homotopy equivalence \cite[\S 5]{bauer2022unified}.
\end{proof}

\subsection{The Persistent Nerve Theorem}
As observed in \cite[Lemma 3.4] {chazal2008towards}, \cite[\S
  3]{botnan2015approximating}, and \cite[Theorem 5.9]{bauer2022unified}, \cref{Thm:Weil_Nerve}
extends readily to a result for diagrams of spaces, the \emph{persistent nerve theorem}.  For brevity's sake, we explicitly state only the extension of \cref{Thm:Weil_Nerve}\,(i).  

To prepare for the statement, we extend the definitions used in the statement of
\cref{Thm:Weil_Nerve} to the diagrammatic setting.

\begin{definition}[Diagrammatic Covers and
    Nerves]\mbox{}\label{Def:Diagrammatic_Covers_Nerves} 
\begin{enumerate}[(i)]
\item For $\C$ a category and $F\colon \C\to \Top$ a functor, a \emph{cover} of $F$
  is a set $U$ of functors from $\C$ to $\Top$ such that  
\begin{enumerate}
\item for each $c\in \ob \C$, \[U_c :=\{G_c \mid G\in U\}\] is a cover of $F_c$,
\item for each $G\in U$ and $\varphi\in \hom_\C(c,d)$, $G_{\varphi}$
  is the restriction of $F_{\varphi}$ to $G_c$.
\end{enumerate}
\item We say a cover $U$ of a functor $F\colon \C\to \Top$ is \emph{weakly good} if for
  each $c\in \ob \C$, $U_c$ is a weakly good cover. 
\end{enumerate}
\end{definition}

The nerve construction for a cover of spaces
(\cref{Def:Diagramatic_Nerve}\,(i)) is functorial, in the sense that 
a cover $U$ of a functor $F\colon \C\to \Top$ induces a ``nerve diagram''
$N(U)\colon \C\to \Simp$, where \[N(U)_c=N(U_c)\] for each $c\in
\ob \C$, and the internal maps of $N(U)$ are the
obvious ones.  Note that these internal maps are always
monomorphisms, regardless of whether the internal maps of $F$ are
injections.  This implies that the subdivision functor $\Sd(N(U))$ of \cref{Def:Subdivision_Cech_Rips}\,(ii) is well defined, which will be needed for the statement of \cref{Thm:General_Multicover_Nerve}.  

\begin{theorem}[Persistent Nerve Theorem {\cite[Theorem 5.9]{bauer2022unified}}]\label{Thm:Persistent_Nerve}
If $U$ is a weakly good open cover of $F\colon \C\to \Top$, then $N(U)$ is weakly equivalent to $F$.
\end{theorem}

\begin{proof}
It is straightforward to check that in the proof of the nerve theorem outlined above, the construction of the space $Z$ extends to yield a functor $Z\colon \C\to \Top$, and the maps $\rho^1$ and $\rho^2$ extend to weak equivalences 
\[N(U)\xleftarrow{\,\rho^1} Z\xrightarrow{\rho^2} F.\qedhere\]
\end{proof}

\subsection{The Multicover Nerve Theorem}
We now extend the persistent nerve theorem to a multicover version.  As with the persistent nerve theorem, one has two natural formulations of the result, extending \cref{Thm:Weil_Nerve}\,(i) and \cref{Thm:Weil_Nerve}\,(ii), respectively; again, for brevity's sake, we explicitly state only the former.  We first prove a version of the multicover nerve theorem for spaces (\cref{Thm:Space_Multicover_Nerve}), and then observe that this extends readily to diagrams of spaces (\cref{Thm:General_Multicover_Nerve}).  

Given a topological space $X$ and a cover $U$ of $X$, we define a
filtration 
\[
\begin{aligned}
&\M(U)\colon (0,\infty)^{\op}\to \Top,\\
&\M(U)_k=\{y\in X\mid y \textup{ is contained in at least $k$
  elements of $U$}\}.
  \end{aligned}
\]

\begin{theorem}[Multicover Nerve Theorem for Spaces]\mbox{}\label{Thm:Space_Multicover_Nerve}
If $U$ is a weakly good open cover of a topological space, %and $\M(U)$ is objectwise paracompact, 
then
$\M(U)$ is weakly equivalent to $\Sd(N(U))$. 
\end{theorem}

\begin{proof}[Proof of \cref{Thm:Space_Multicover_Nerve}]
For $k\in (0, \infty)$, let $P^{\geq k}$ denote the subposet of $P^U$
consisting of faces of dimension at least $k-1$, and let $j^k\colon
P^{\geq k}\to P^U$ denote the inclusion.  By
\cref{Rem:Included_Maps_HCLs}, the spaces \[Z_k:=\hocolim_{P^{\geq k}}
(D^U\circ j^k)\] assemble into a functor $Z\colon(0,\infty)^{\op}\to \Top$, where
each internal map $Z_{k,l}$ is induced by the inclusions of posets $P^{\geq k}\hookrightarrow P^{\geq l}$.
 
To prove the theorem, we will show that there exist objectwise
weak homotopy equivalences   
\[
\Sd(N(U)) \xleftarrow{\,\rho^1} Z \xrightarrow{\rho^2}  \M(U).
\]
Let $\ast \colon P^{\geq k}\to \Top$ denote the constant functor to a 
point.  Since $U$ is a weakly good cover, we have an objectwise weak homotopy
equivalence $D^U\circ j^k\to \ast$ which,  by 
\cref{Prop:Homotopy_Invariance_Of_Homotopy_Colimit},  induces a
weak homotopy equivalence between the homotopy colimits of the two
diagrams.  Arguing as in \cref{Eq:Nerve_As_Hocolim}, we see that $\hocolim(*)$ is canonically homeomorphic to $\Sd(N(U))_k$.  Therefore, 
we obtain a weak homotopy equivalence 
\[
\rho^1_k\colon \hocolim_{P^{\geq k}}(D^U\circ j^k)\to \Sd(N(U))_k.
\]
Note that $\colim_{P^{\geq k}} (D^U\circ j^k)=\M(U)_k$; we define 
$\rho^2_k\colon Z_k\to \M(U)_k$ to be the natural map from the homotopy colimit to the
colimit.  It is easily checked that the maps $\rho^1_k$ and $\rho^2_k$ are natural
with respect to $k$, so they assemble into natural transformations
$\rho^1$ and $\rho^2$.  

To finish the proof, we need to check that each $\rho^2_k$ is a weak homotopy 
equivalence; it is this check that is omitted in ~\cite[Appendix
  B]{cavanna17when}.  We have already seen in the proof of \cref{Thm:Weil_Nerve} that $\rho^2_1$ 
is a weak homotopy equivalence, so it remains to handle the case $k>1$.  
Our argument is inspired by Sheehy's
first proof of his version of the multicover nerve theorem,
specifically~\cite[Lemma 8]{sheehy2012multicover}.

Let $P^k\subset P^U$ denote the set of $(k-1)$-faces of $N(U)$, i.e.,
the order-$k$ subsets of $U$ with non-empty common intersection.  $P^k$
indexes a cover $U^k=\{U^k_a\}_{a\in P^k}$ of $\M(U)_k$, with
$U^k_a=D^U_a$.  Via the bijection $P^k\to U^k$, we may identify each
element of $P^{U^k}$ with a set of order-$k$ subsets
of $U$.  We have a functor 
\[
F\colon P^{U^k}\to P^{\geq k}
\]
specified by the assignment
\[
\{t_1,\ldots,t_m\} \mapsto t_1\cup \cdots \cup t_m,
\] 
where each $t_i$ is an order-$k$ subset of $U$, and the target is an
order-$k'$ subset of $U$, for some $k'\geq k$.  Note that
\[D^{U^k}=D^U\circ j^k\circ F,\] since the intersection of intersections of sets is the intersection of all the sets involved.

We have a map 
\[
q\colon \hocolim_{P^{U^k}} (D^{U^k})\to \colim_{P^{U^k}}(D^{U^k})=\M(U)_k.
\]
Since $U^k$ is an open cover of $\M(U)_k$, the proof that $p_1^2$ is a weak homotopy equivalence also establishes that $q$ is a weak homotopy equivalence.  Therefore, by the 2-out-of-3
property, to show that
$\rho_k^2$ is a weak homotopy equivalence,  it suffices to check that 
\begin{enumerate}
\item the map $q$ factors as $q = \rho_k^2\circ F_*$, where  
\[
F_*\colon \hocolim_{P^{U^k}}(D^U\circ j^k\circ F)\to \hocolim_{P^{\geq k}} (D^U\circ j^k)
\] 
is the map induced by $F$, and

\item $F_*$ is a homotopy equivalence.
\end{enumerate}

The first statement follows readily from the definitions of the three
maps.  We establish the second statement using a homotopy finality
argument: Note that for each $q\in P^{\geq k}$, there is a unique
minimum element $p$ of $P^{U^k}$ with $F(p)=q$, namely the set
of all order-$k$ subsets of $q$.  The result then follows from
\cref{Prop:Min_Elt_Contractible}\,(iii) and \cref{Thm:Homotopy_Finality}. 
\end{proof}

Finally, we extend the multicover nerve theorem for spaces to a version for diagrams of spaces which refines the persistent nerve theorem (\cref{Thm:Persistent_Nerve}).  Given a functor $F\colon \C\to \Top$ and a cover $U$ of $F$ (\cref{Def:Diagrammatic_Covers_Nerves}), we define the functor 
\[
\begin{aligned}
&\M(U)\colon (0,\infty)^{\op}\times \C\to \Top,\\
&\M(U)_{(k,r)}=\{y\in F_r\mid y \textup{ is contained in at least $k$ elements of $U_r$}\}.
\end{aligned}
\]

\begin{theorem}[Multicover Nerve Theorem for Diagrams of Spaces]\mbox{}\label{Thm:General_Multicover_Nerve}
If $U$ is a weakly good open cover of a functor $F:\C\to \Top$, then $\M(U)$ is weakly equivalent to $\Sd(N(U))$.
\end{theorem}

\cref{Thm:Multicover}\,(i) is the special case of this theorem where
$\C=(0,\infty)$ and $U_r$ is a set of open balls of radius $r$ with fixed centers.

\begin{proof}[Proof of \cref{Thm:General_Multicover_Nerve}]
In close analogy with the proof of \cref{Thm:Persistent_Nerve}, it is easy to check that the maps $\rho_k^1$ and
$\rho_k^2$ constructed in the proof of
\cref{Thm:Space_Multicover_Nerve} with respect to the cover $U_r$ of
$F_r$ are natural not only with respect to $k$ but also with respect
to $r$.  Thus, the proof of \cref{Thm:Space_Multicover_Nerve} extends
to a proof of this result.
\end{proof}

\begin{remark}[Multicover Nerve Theorem for Closed Covers]\label{Rem:Multicover_Closed}
In our proof of \cref{Thm:General_Multicover_Nerve}, the assumption that cover elements are open is needed only to establish that the maps $q$ defined in the proof of \cref{Thm:Space_Multicover_Nerve} are weak homotopy equivalences.  In fact, \cite[Proposition 5.37]{bauer2022unified} tells us that the maps $q$ are also weak homotopy equivalences for a large class of closed covers, namely those satisfying the conditions of \cite[Theorem 5.9.1.b]{bauer2022unified}.  Therefore, the multicover nerve theorem also holds for such covers, provided the covers are also weakly good.    
\end{remark}

\appendix

\section{A Computational Example of the Stability of Degree-Rips Bifiltrations\label{Sec:Examples}}
\label{Sec:RIVET_Point_Clouds}

In this section, we explore the stability of the degree-Rips bifiltration in an example, using the 2-parameter persistence software RIVET
\cite{lesnick2015interactive,lesnick2019computing,rivet}.  We consider three point clouds $X$, $Y$, and $Z$, shown in Figure
\ref{Fig:Point_Clouds}:
\begin{itemize}
\item $X$ consists of 475 points
sampled uniformly from an annulus in $\R^2$ with outer radius .5 and inner radius .4. 
\item $Y=X\cup N$, where $N$ consists of 25 points sampled uniformly from a disc of radius 
.4.
\item $Z$ consists of 500 points sampled uniformly from a disc of radius .5.
\end{itemize}
We would like to consider, for each $W\in \{X,Y,Z\}$ the homology module $H_1(\nDRips(W))$ with coefficients in $\Z/2\Z$.  However, working directly with $H_1(\nDRips(W))$ is  computationally expensive, so we instead work with an approximation $H(W)$ of $H_1(\nDRips(W))$; this is explained in \cref{Sec:Bifiltrations_We_Compute}.  In \cref{Sec:Visualization}, we illustrate the stability of degree bifiltrations in practice by using RIVET to visualize invariants of $H(X)$, $H(Y)$, and $H(Z)$.  Then, in \cref{Sec:Stability_Analysis}, we apply the stability result \cref{Prop:Degree_Subspaces}\,(ii) to explain a small part of the similarity between $H(X)$ and $H(Y)$ observed in our visualizations.  Recall that \cref{Prop:Degree_Subspaces}\,(ii) applies to a nested pair of data sets; in \cref{Rem:Symmetric_Version_No_Good}, at the end of this section, we observe that \cref{Thm:Degree_Rips_Stab}\,(ii), which does not assume that the data is nested, does not constrain the similarity between $H(X)$ and $H(Y)$.

We warn the reader that the remainder of this section is somewhat technical, in part because of the approximations involved.  We invite the reader to skim the section on a first reading, focusing on understanding the figures.

\begin{figure*}[t!]
    \centering
    \begin{subfigure}[t]{0.33\textwidth}
        \centering
        \includegraphics[width=0.95\textwidth]{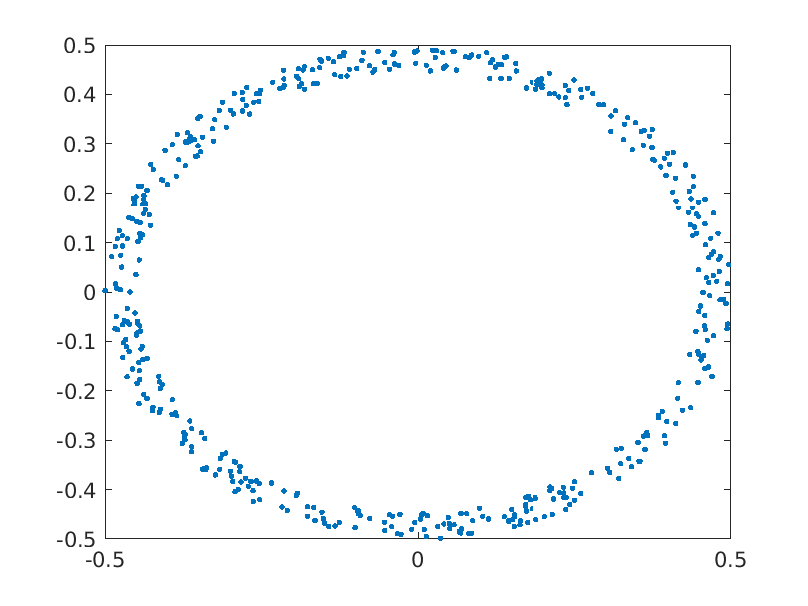}
        \caption{$X$}
    \end{subfigure}%
    ~ 
    \begin{subfigure}[t]{0.33\textwidth}
        \centering
        \includegraphics[width=0.95\textwidth]{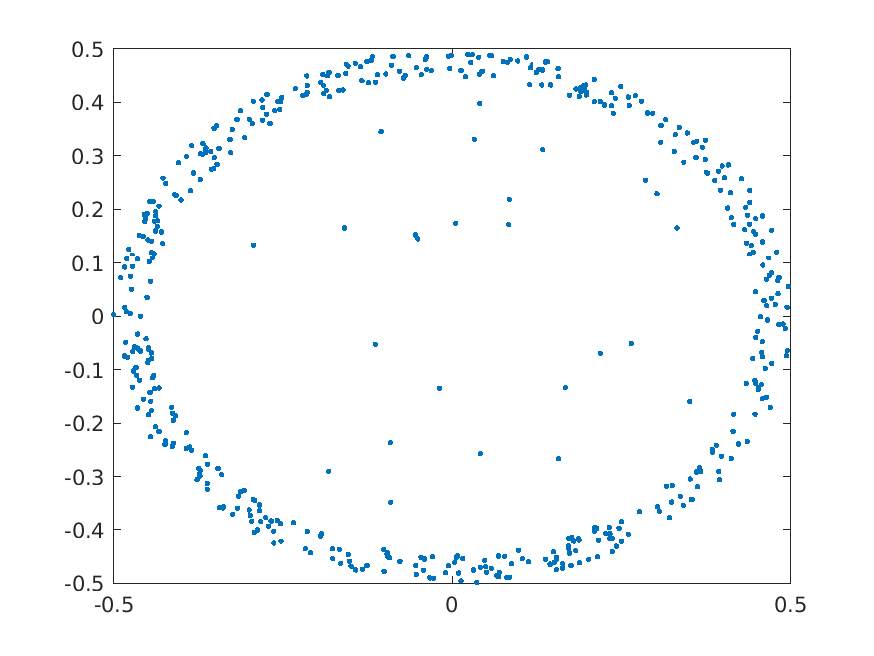}
        \caption{$Y$}
    \end{subfigure}
    ~
        \begin{subfigure}[t]{0.33\textwidth}
        \centering
        \includegraphics[width=0.95\textwidth]{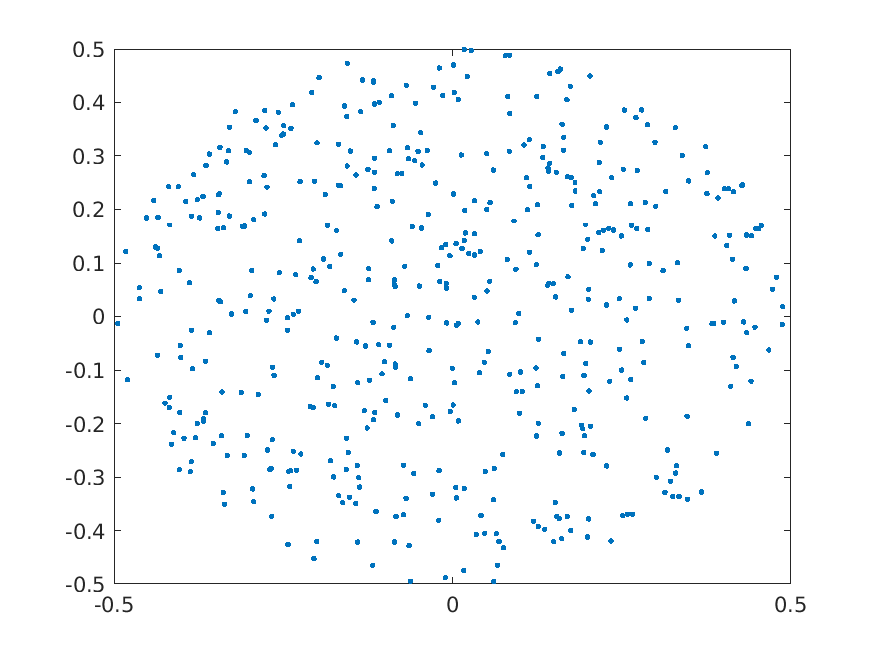}
        \caption{$Z$}
    \end{subfigure}
    
    \caption{The point clouds $X$, $Y$, $W$.} 
        \label{Fig:Point_Clouds}
\end{figure*}

\subsection{Approximations to the Degree-Rips Bifiltrations}\label{Sec:Bifiltrations_We_Compute}
For discussing RIVET computations, it is convenient to introduce a variant of the normalized Degree-Rips bifiltration.  First, for $W$ a finite metric space, let $\ClR(W)\colon [0,\infty)\to \Top$ be the filtration defined by taking \[\ClR(W)_r=\lim_{r<s} \Rips(W)_s.\]  
This is precisely the variant of the Rips construction mentioned in \cref{Rem:Different_Def_Of_Rips_Cech}.  

Now define a bifiltration 
\[\ClDR(W)\colon \R^{\op}\times [0,\infty)\to \Simp\] by taking $\ClDR(W)_{(k,r)}$ to be the maximal subcomplex of $\ClR(W)_r$ whose vertices have degree at least $|W|(r-1)$.  This bifiltration is slightly different from the normalized degree-Rips bifiltration $\nDRips(W)$ defined in \cref{sec:bifilt}---for one thing, they are indexed by different posets---but it's not hard to see that these two bifiltrations are equivalent, in the sense that each determines the other in a simple way.  Moreover, the restriction of $\ClDR(W)$ to the poset $\OurPoset$ is $\epsilon$-interleaved with $\nDRips(W)$ for any $\epsilon>0$.  

The largest simplicial complex in $\ClDR(W)$ is the simplex with vertices $W$; denote this as $S$.  For any simplex $\sigma \in S$, we define the \emph{set of bigrades of appearance of $\sigma$} to be  the set of minimal elements $(k,r)\in \R^{\op}\times [0,\infty)$ such that $\sigma \in  \ClDR(W)_{(k,r)}$.  This is a finite and nonempty subset of $\OurPoset$.  

To control the cost of the computations, for each $W\in \{X,Y,Z\}$ we in fact work with a 
``coarsening'' \[F(W)\colon \R^{\op}\times [0,\infty)\to \Top\] of the bifiltration $\ClDR(W)$, where the
bigrades of appearance of all simplices are rounded upwards so as to lie on
a uniform $100\times 100$ grid; the precise definition of this coarsening is given in \cite{lesnick2015interactive}.  The grid is 
chosen in a way that ensures that $\ClDR(W)$ and $F(W)$ are
$(\tau^{\frac{1}{100}},\Id)$-interleaved.  Let $F'(W)$ denote the restriction of $F(W)$ to $\OurPoset$.  By the generalized triangle inequality for interleavings (\cref{Rem:Strict_Variants}),  $\nDRips(W)$ and $F'(W)$ are  $(\tau^{\frac{1}{100}+\epsilon},\tau^\epsilon)$-interleaved for all $\epsilon>0$.

For $W\in \{X,Y,Z\}$, let $H(W)$ denote the homology module $H_1(F(W))$ with coefficients in $\Z/2\Z$.  Observe that $H(W)$ is finitely presented.

\subsection{Visualization}\label{Sec:Visualization}

RIVET computes and visualizes three invariants of a finitely presented biperistence module $M\colon\R^{\op}\times [0,\infty)\to \Top$: 
\begin{itemize}
\item The Hilbert function of $M$, i.e., the dimension of each vector space of $M$.
\item The bigraded Betti numbers of $M$; these are functions $\beta_i^M\colon\R^2\to \Nbb$, $i\in \{0,1,2\}$ which, roughly speaking, record the birth indices of generators, relations, and relations among the relations.
\item The  \emph{fibered barcode} of $M$ \cite{cerri2013betti}, i.e., the map sending each affine line $\ell\subset \R^{\op}\times [0,\infty)$ of non-positive slope to the barcode $\B{M^{\ell}}$, where $M^{\ell}$ is the restriction of $M$ to $\ell$.  We regard $\B{M^{\ell}}$ as a collection of intervals on the line $\ell$.
\end{itemize}
See \cite{lesnick2015interactive,lesnick2019computing} for more details about these invariants and their computation in RIVET.

\cref{Fig:RIVET_Plots} shows RIVET's visualization of $H(X)$, $H(Y)$, and $H(Z)$, and of the barcodes $\protect\B{H(X)^{\ell}}$, $\protect\B{H(Y)^{\ell}}$, and $\protect\B{H(Z)^{\ell}}$  for one choice of $\ell$.
To explain the figure, first note that the $x$-axis is mirrored in each subfigure, relative to the usual convention, so that values decrease from left to right.   In each figure, the Hilbert function is represented by grayscale shading, where the darkness is proportional to the homology dimension.  The lightest non-white shade of gray shown in each figure corresponds to a value of $1$.  The bigraded Betti numbers are represented by translucent colored dots whose area is proportional to the value.  The $0^{\mathrm{th}}$, $1^{\mathrm{st}}$, and $2^{\mathrm{nd}}$ bigraded Betti numbers are shown in green, red, and yellow, respectively.  In each figure, the line $\ell$ is shown in blue, and the corresponding barcode is plotted in purple, with each interval offset perpendicularly from the line.

Note that the Hilbert function of $H(X)$ takes value 1 on a large connected region parameter space, whose restriction to the left half-plane $k> 0$ looks roughly like a triangle.  The Hilbert function of $H(Y)$ also takes value 1 on a substantial region in parameter space, albeit one smaller than for $H(X)$.  In contrast, the Hilbert function of $H(Z)$ does not take the value 1 in a large region of the parameter space, and in fact almost all of the support of $H(Z)$ lies very near $(0,0)$.

\begin{figure*}[ht!]
    \centering
        \begin{subfigure}[t]{0.5\textwidth}
        \centering
        \includegraphics[width=0.95\textwidth]{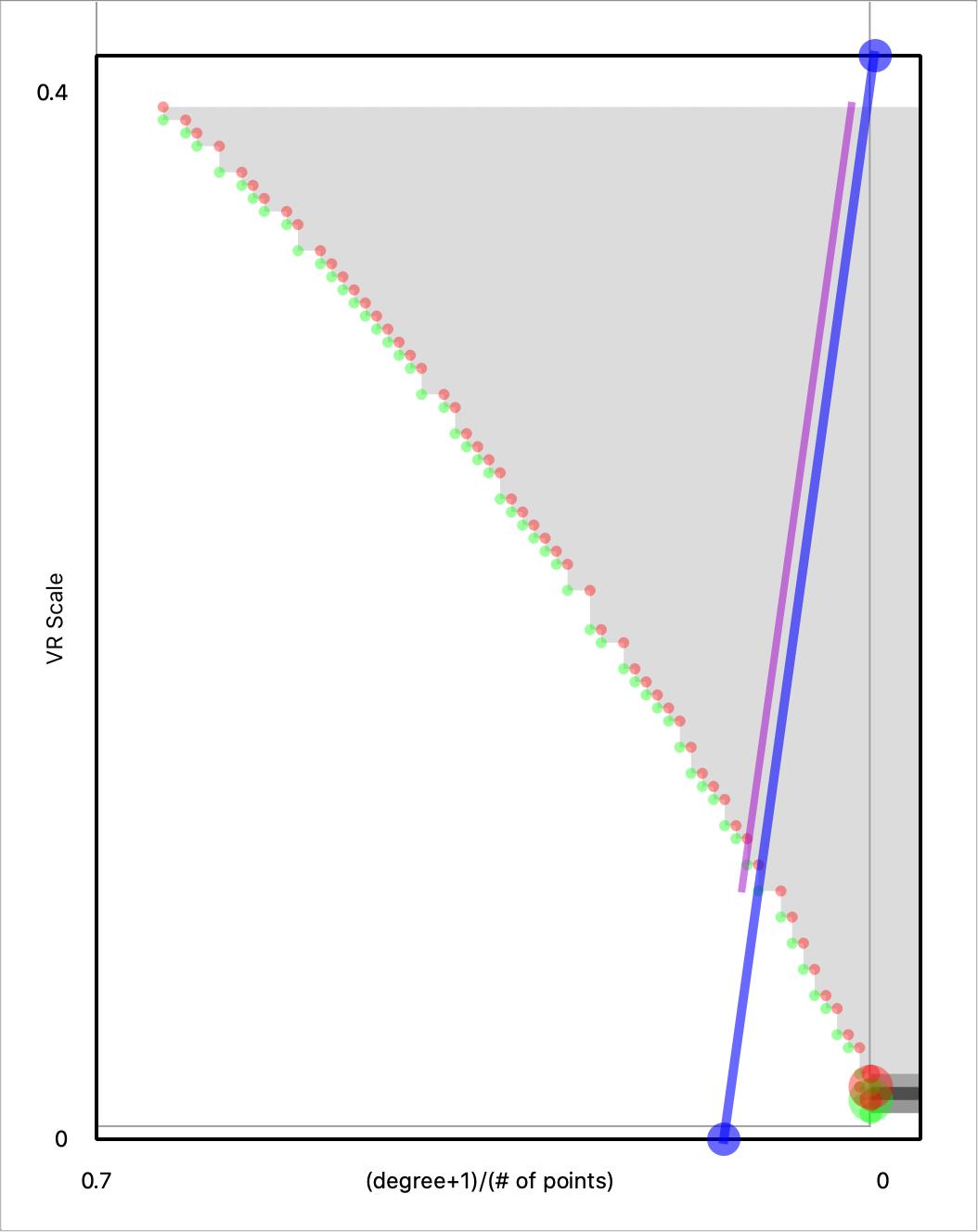}
        \caption{$\protect \B{H(X)^{\ell}}$}
    \end{subfigure}%
    ~ 
    \begin{subfigure}[t]{0.5\textwidth}
        \centering
        \includegraphics[width=0.95\textwidth]{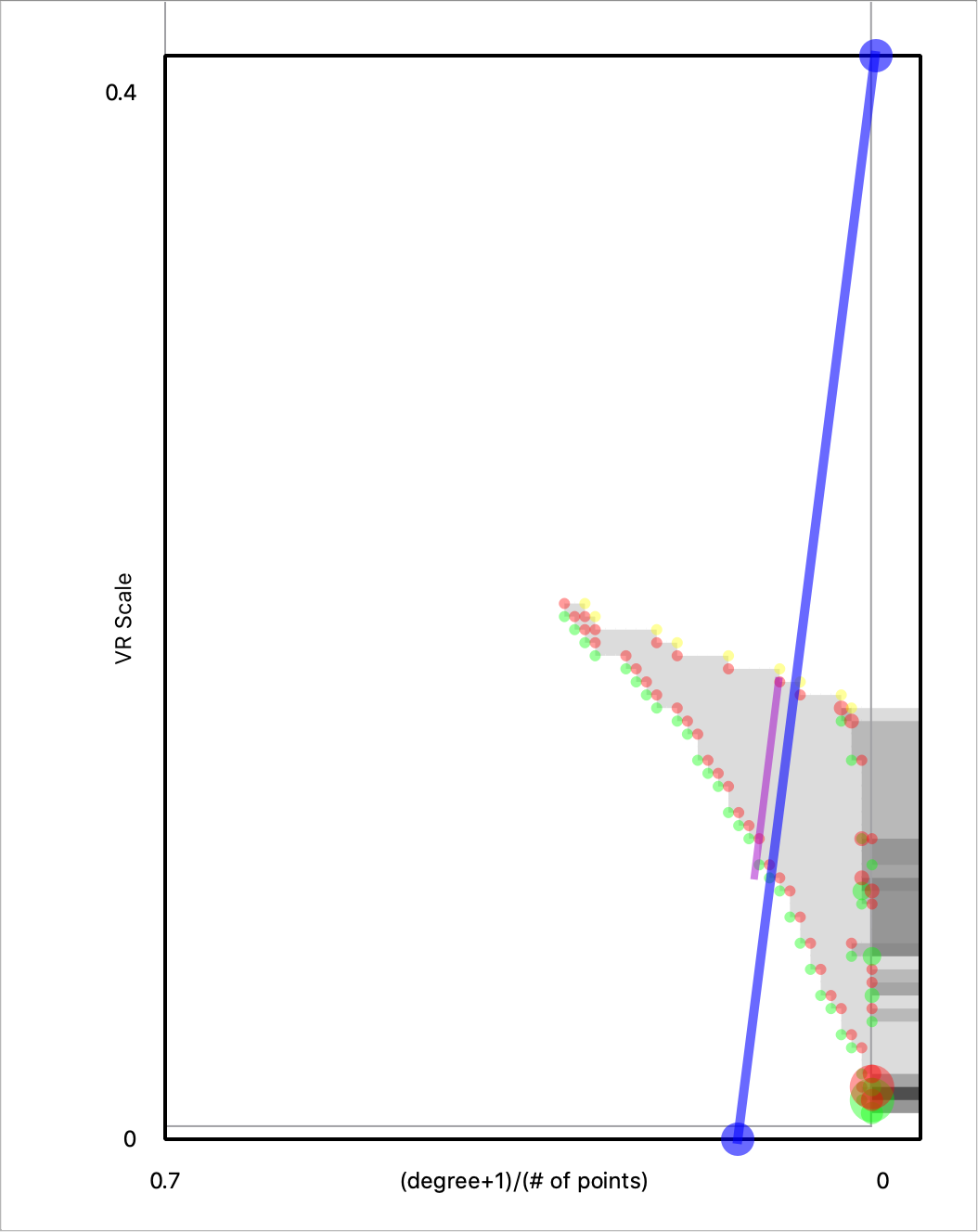}
        \caption{$\protect \B{H(Y)^{\ell}}$}
            \end{subfigure}
      \par\bigskip
         \begin{subfigure}[t]{0.5\textwidth}
        \centering
        \includegraphics[width=0.95\textwidth]{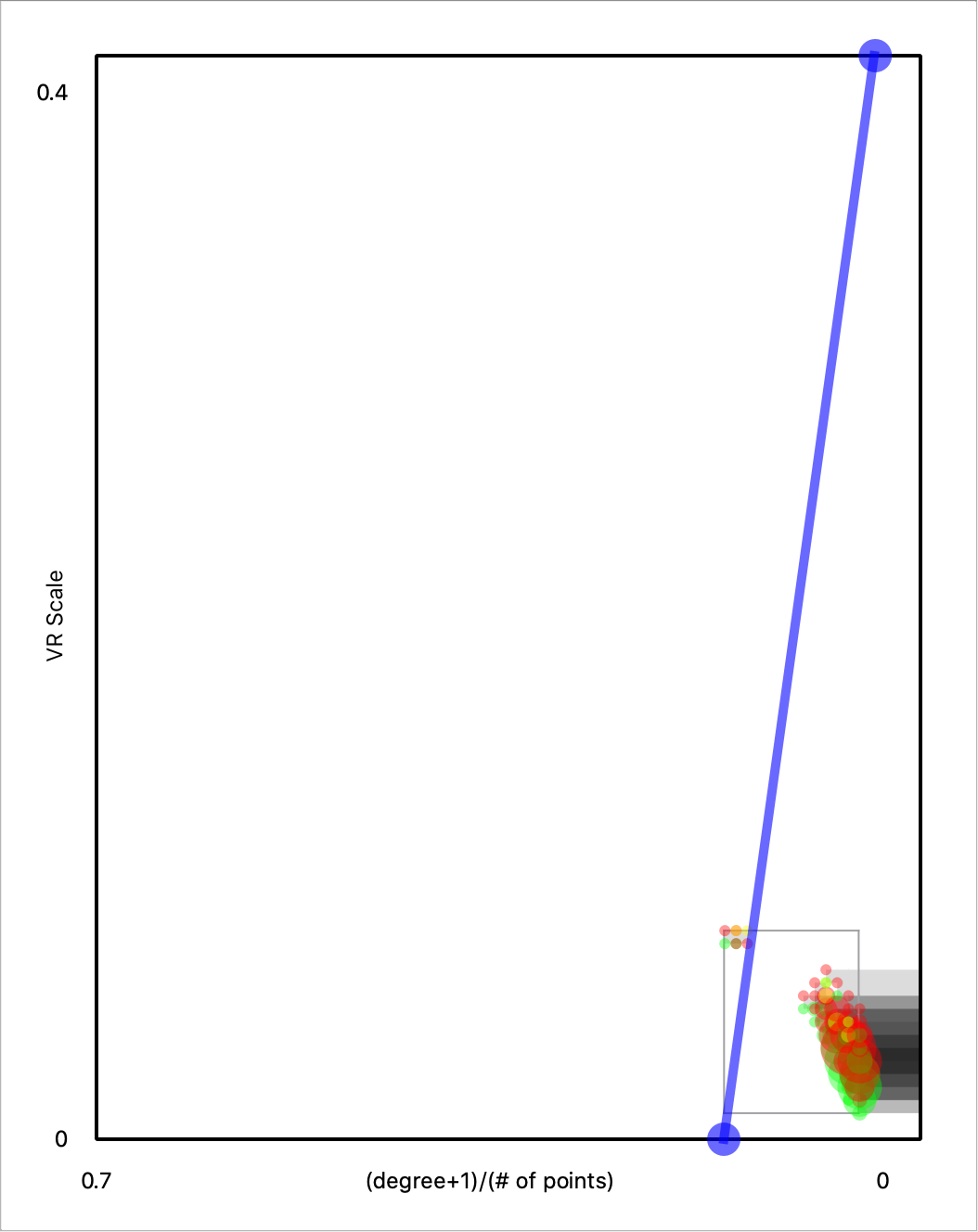}
        \caption{$\protect \B{H(Z)^{\ell}}$}
    \end{subfigure}
    \caption{RIVET's visualization of $H(X)$, $H(Y)$, $H(Z)$ and the barcodes $\protect\B{H(X)^{\ell}}$, $\protect\B{H(Y)^{\ell}}$, $\protect\B{H(Z)^{\ell}}$,  for one choice of line $\ell$.  See the text for an explanation.} 
        \label{Fig:RIVET_Plots}
\end{figure*}

\subsection{Stability Analysis}\label{Sec:Stability_Analysis}
By appealing to \cref{Prop:Degree_Subspaces}\,(ii), we can explain a small part of the similarity between the structures of $H(X)$ and $H(Y)$ observed in \cref{Fig:RIVET_Plots}.  More specifically, we will show that given $H(X)$, \cref{Prop:Degree_Subspaces}\,(ii) implies that the Hilbert function of $H(Y)$ has non-trivial support on a small region of parameter space.  The only property of $Y$ we use in our analysis is that $Y$ is a metric space of cardinality 500 containing $X$ as a subspace.

By \cref{Prop:Degree_Subspaces}\,(ii),  $\nDRips(X)$ and $\nDRips(Y)$ are $(\kappa^{|X|/|Y|},\gamma^{\delta})$-homotopy interleaved for any $\delta>\tfrac{25}{500}=\tfrac{1}{20}$.  Note that $|X|/|Y|=\tfrac{475}{500}=\tfrac{19}{20}$.  By the discussion of \cref{Sec:Bifiltrations_We_Compute}, $\nDRips(X)$, $F'(X)$ are $(\tau^{\frac{1}{100}+\epsilon},\tau^\epsilon)$-interleaved for all $\epsilon>0$, and the same is true for $\nDRips(Y)$, $F'(Y)$.   For $\epsilon\geq 0$, let $\zeta^\epsilon$ be the forward shift given by \[\zeta^{\epsilon}(x,y)=\left(\frac{19x}{20}-\frac{1}{100}-\epsilon,y+\frac{1}{100}+\epsilon\right),\]
and let $\zeta=\zeta^0$.  
A strict interleaving is also a homotopy interleaving, so by \cref{Prop:Triangle_Inequality}, $F'(X)$ and $F'(Y)$ are $(\zeta^{\epsilon},\gamma^{\frac{6}{100}+\epsilon})$-homotopy interleaved for all $\epsilon>0$.  From this, one can show that in fact, $F'(X)$ and $F'(Y)$ are $(\zeta,\gamma^{\frac{6}{100}})$-homotopy interleaved, using an argument similar to the proof of \cite[Theorem 6.1]{lesnick2015theory}.

Letting $H'(X)$ and $H'(Y)$ denote the respective restrictions of $H(X)$ and $H(Y)$ to $\OurPoset$ (i.e., $H'(X)=H_1(F'(X))$ and $H'(Y)=H_1(F'(Y))$), we then have that $H'(X)$ and $H'(Y)$ are  $(\zeta,\gamma^{\frac{6}{100}})$-interleaved by \cref{Proposition:Homotopy_to_Homology_Interleavings}.  In what follows, we will show that this constrains certain vector spaces in $H'(Y)$ to have dimension at least one.   

Let $\Delta$ denote the large triangle-like connected region in $\OurPoset$ where the Hilbert function of $H'(X)$ takes value 1.  The boundary of $\Delta$ intersects the vertical line $k=0$ and the horizontal line \[r=\tfrac{7891389}{20000000}\approx.396.\]  By inspecting the bigraded Betti numbers and fibered barcode of $H(X)$, as shown in  \cref{Fig:RIVET_Plots} and \cref{Fig:RIVET_Zoom}, it can be seen that $\Delta$ is in fact is contained in the support of a thin indecomposable summand of $H'(X)$.   Thus, if $a\leq b\in \Delta$ (with respect to the partial order on $\R^{\op}\times \R$), then $\rank({H'(X)}_{a,b})=1$. 
\begin{figure*}[ht!]
    \centering
        \includegraphics[width=0.475\textwidth]{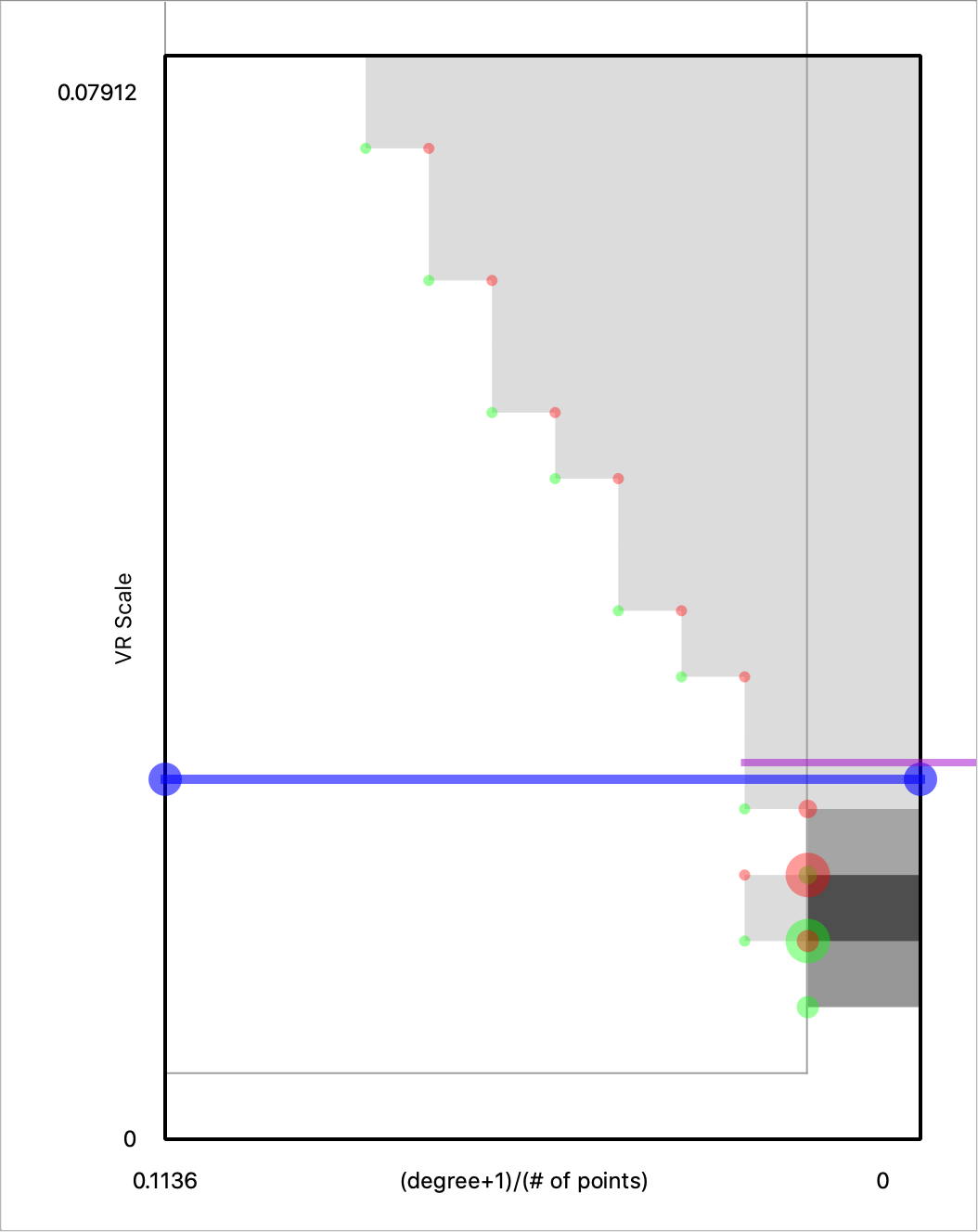}
    \caption{RIVET's visualization of $H(X)$, zoomed in near $(0,0)$.} 
        \label{Fig:RIVET_Zoom}
\end{figure*}
In particular, if $(k,r)\in \Delta$ and also \[\gamma^{\frac{6}{100}}\circ
\zeta(k,r)=\left(\frac{19k}{20}-\frac{7}{100},3r+\frac{9}{100}\right)\in \Delta,\] then \[\rank
H'(X)_{(k,r),\gamma^{\frac{6}{100}}\circ
\zeta(k,r)}=1.\]  Now if $(k,r)\in \Delta$, then
$(\frac{19k}{20}-\frac{7}{100},3r+\frac{9}{100})\in \Delta$ if and only if $k>c_x$ and
$r<c_y$, where 
\[c_x=\tfrac{7}{95}\approx .074,\quad c_y=\left(\tfrac{2030463}{20000000}\right)\approx .102.\]  From the output of RIVET, it can be seen that the set \[\Omega=\{(k,r)\in \Delta\mid k>c_x,\ r<c_y\}\] is a small but non-empty connected subregion of
$\Delta$, containing the grades of three elements in any minimal set of
generators for $H'(X)$.  $\Omega$ is shown in \cref{Fig:Highlights}\,(A).  Explicitly, 
 \[\Omega=\rect(v_1)\cup \rect(v_2)\cup \rect(v_3),\]
where
\begin{align*}
\rect(a)&=\{(x,y)\mid a_1\geq x > c_x,\ a_2\leq y < c_y\},\\
v_1&=\left(\tfrac{2657}{23750},\tfrac{1897929}{20000000}\right)\approx(.112,.095),\\
v_2&=\left(\tfrac{2183}{23750},\tfrac{1698147}{20000000}\right)\approx(.092,.085),\\
v_3&=\left(\tfrac{973}{11875},\tfrac{299673}{4000000}\right) \approx(.082,.075).
\end{align*}
\begin{figure*}[ht!]
    \centering
        \begin{subfigure}[t]{0.5\textwidth}
        \centering
        \includegraphics[width=0.95\textwidth]{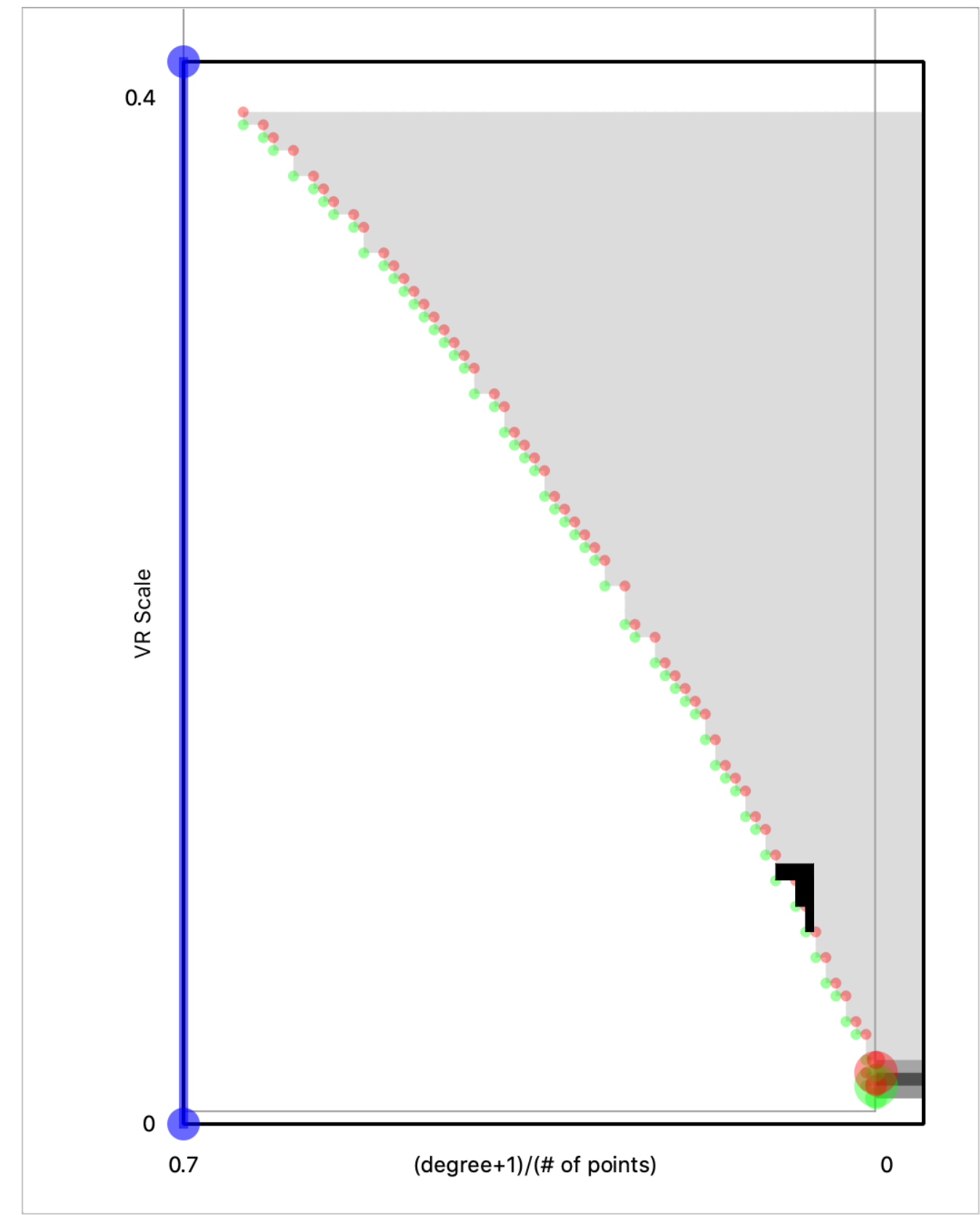}
        \caption{$\Omega$}
    \end{subfigure}%
    ~ 
    \begin{subfigure}[t]{0.5\textwidth}
        \centering
        \includegraphics[width=0.93\textwidth]{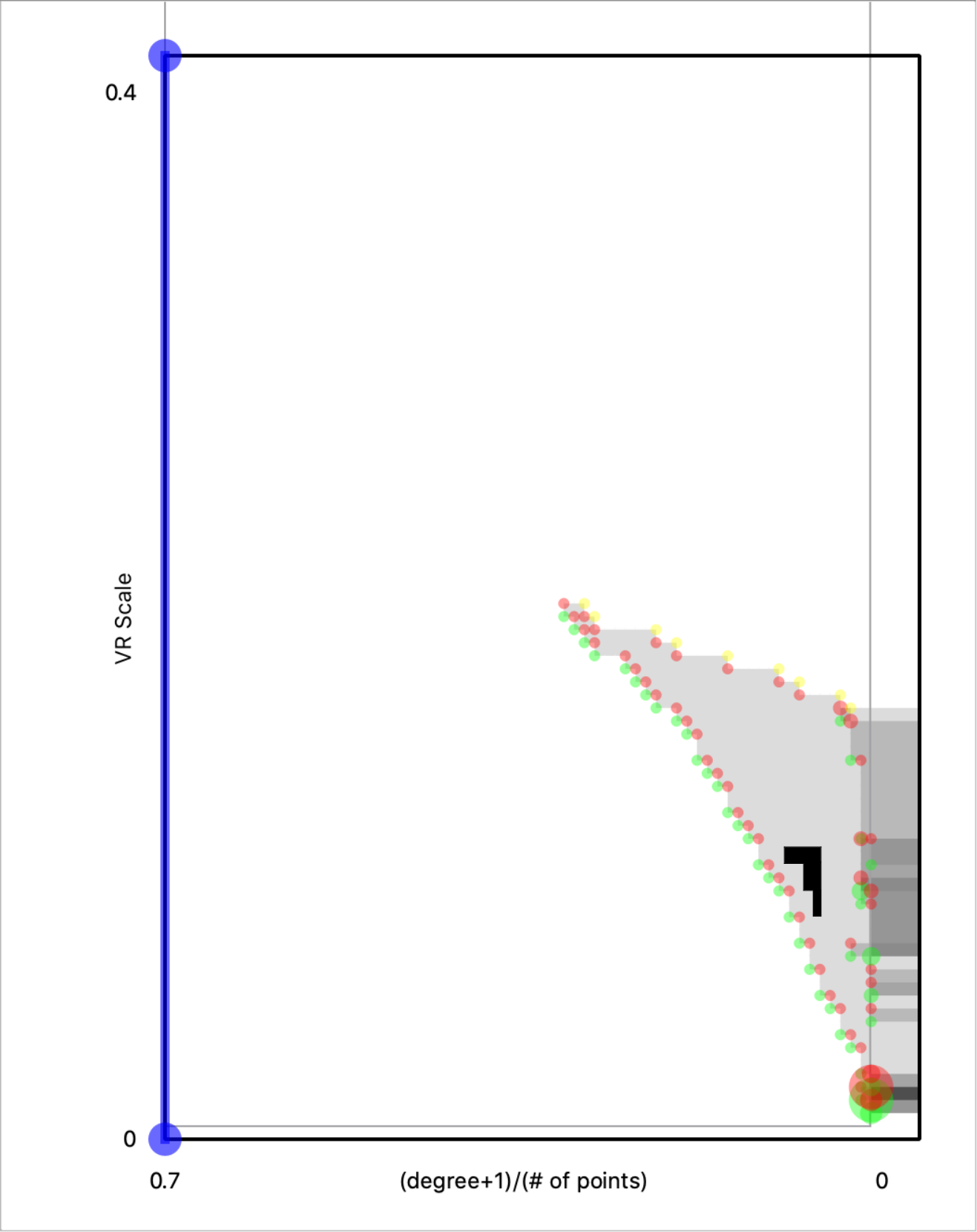}
        \caption{$\zeta(\Omega)$}
            \end{subfigure}
    \caption{The regions $\Omega$ and $\zeta(\Omega)$, shown in solid black.  [Note: These regions were drawn by hand in software, and so are not as precise as if they had been drawn algorithmically.  However, the imprecisions are miniscule.]} 
        \label{Fig:Highlights}
\end{figure*}

For any $(k,r)\in \Omega$, a $(\zeta,\gamma^{\frac{6}{100}})$-interleaving between $H'(X)$ and $H'(Y)$ provides a factorization of the non-zero linear map $H'(X)_{(k,r),(\frac{19k}{20}-\frac{7}{100},3r+\frac{9}{100})}$ through \[H'(Y)_{\zeta(k,r)}=H'(Y)_{(\frac{19k}{20}-\frac{1}{100},r+\frac{1}{100})}.\]  Therefore, the support of the Hilbert function of $H'(Y)$ must contain $\zeta(\Omega)$.   $\Omega$ is shown in \cref{Fig:Highlights}\,(B).  Letting \[d_x=\tfrac{6}{100},\quad d_y=\tfrac{2230463}{20000000}\approx .112,\]
$\zeta(\Omega)$ can be written explicitly as 
\[\zeta(\Omega)=\rect(w_1)\cup \rect(w_2)\cup \rect(w_3),\]
where 
\begin{align*}
\rect(a)&=\left\{(x,y)\mid a_1\geq x > d_x ,\ a_2\leq y < d_y\right\},\\
w_1&=\left(\tfrac{2407}{25000},\tfrac{2097929}{20000000}\right)\approx(.096,.105),\\
w_2&=\left(\tfrac{1933}{23750},\tfrac{1898147}{20000000}\right)\approx(.077,.095),\\
w_3&=\left(\tfrac{212}{3125},\tfrac{339673}{4000000}\right) \approx(.068,.085).
\end{align*}

  \cref{Fig:Highlights}\,(B) indicates that the support of the Hilbert function of $H'(Y)$ does indeed contain $\zeta(\Omega)$, and in fact is much larger.

\begin{remark}
We have shown that given $H(X)$,  \cref{Prop:Degree_Subspaces}\,(ii) constrains the structure of $H(Y)$.  Using a similar argument, one can show that given $H_1(\nDRips(X))$,  \cref{Prop:Degree_Subspaces}\,(ii) constrains the structure of $H_1(\nDRips(Y))$.  However, a similar argument also shows that given $H(Y)$,  \cref{Prop:Degree_Subspaces}\,(ii) provides no constraint on $H(X)$; and similarly, given $H_1(\nDRips(Y))$, \cref{Prop:Degree_Subspaces}\,(ii) provides no constraint on $H_1(\nDRips(X))$.
\end{remark}

\begin{remark}\label{Rem:Symmetric_Version_No_Good}
By \cref{Rem:Outliers}, $\GPr(X,Y)\leq d_P(X,Y)\leq\frac{25}{500}=\frac{1}{20}$.  Starting from this observation, one can perform a stability analysis similar to the one done above, but using the weaker interleaving provided by \cref{Thm:Degree_Rips_Stab}\,(ii) in place of the the one provided by \cref{Prop:Degree_Subspaces}\,(ii).  It is not difficult to check that the interleaving between $H'(X)$ and $H'(Y)$ provided by such an analysis can be taken to be trivial.  Further, using a similar argument, 
 one can show that the interleavings between $H_1(\nDRips(X))$ and $H_1(\nDRips(Y))$ guaranteed to exist by \cref{Thm:Degree_Rips_Stab}\,(ii) can also be taken to be trivial.  Thus, the existence of this interleaving does not constrain the relationship between $H_1(\nDRips(X))$ and $H_1(\nDRips(Y))$ at all.
Because the shared topological signal present in $X$ and $Y$ is especially strong, this indicates that relative to the needs of applications, \cref{Thm:Degree_Rips_Stab} is a rather weak result.
\end{remark}

\begin{remark}\label{Rem:Stronger_Stability}
\cref{Fig:RIVET_Plots} makes clear that while $H(X)$ and $H(Y)$ have
rather different global structure, they do share substantial
qualitative similarities that neither module shares with $H(Z)$.
While our analysis demonstrates that
\cref{Prop:Degree_Subspaces}\,(ii) non-trivially constrains the
relationship between $H(X)$ and $H(Y)$, most of the similarity between
the two modules seen in \cref{Fig:RIVET_Plots} is not explained by the
proposition (or, to the best of our knowledge, by any other known
result).  It would be valuable to develop a refinement of our stability theory which more fully explains the observed similarity.
\end{remark}

\section{Proof of~\cref{Prop:Degree_Tightness}}\label{sec:tightness}

\cref{Prop:Degree_Tightness} is equivalent to the following more concrete statement:
\begin{proposition}\label{Prop:More_Explicit}
For any $c\in [1,3)$,
\begin{enumerate}
\item[(i)]  there exist finite subsets $X$ and $Y$ of a good metric space and $\delta>\Pro(\nu_X,\nu_Y)$
  such that $\nDCech(X)$ and $\nDCech(Y)$ are not $\gamma^{\delta,c}$-homotopy interleaved.
\item[(ii)] there exist finite metric spaces $X$ and $Y$ and $\delta>\GPr(\mu_X,\mu_Y)$
  such that $\nDRips(X)$ and $\nDRips(Y)$ are not $\gamma^{\delta,c}$-homotopy interleaved.
\end{enumerate}
\end{proposition}

We prove statement (ii).  Statement (i) follows from essentially the same argument; we leave the easy adaptation to the reader.  We give a constructive proof of statement (ii), involving sets in a high-dimensional Euclidean space with the $\ell_1$ metric.  In an effort to make the ideas more accessible, we will first present a simpler argument which proves the proposition only for $c\in [1,2)$, 

Fix $c\in [1,2)$, and choose $r>\frac{1}{2(2-c)}\geq\frac{1}{2}$.
Let $\mathbf{e}_i$ denote the $i^{\mathrm{th}}$ standard basis vector in $\R^3$ and let $\vec 0\in \R^3$ denote the zero vector.  Let
\[
Y=\{r\mathbf{e}_i\mid 1\leq i\leq 3\} \quad\textrm{and}\quad  Z=Y\cup\{\vec 0\}.
\]
We regard $Y$ and $Z$ as metric spaces via restriction of the $\ell_1$ metric on $\R^3$.

By \cref{Rem:Outliers}, we
have \[\GPr(\mu_Y,\mu_Z)\leq\Pro(\nu_Y,\nu_Z)\leq \frac{1}{4}.\]  
Choose $\delta$ in the open interval
$\left(\frac{1}{4},\frac{3}{8}\right)$, and choose $\epsilon\in
(0,\frac{r(1-\frac{c}{2})-\delta}{c}]$.

Since $\epsilon>0$, we have 
$\vec 0\in \nDRips(Z)_{(\frac{3}{4},\frac{r}{2}+\epsilon)}$, so $\nDRips(Z)_{(\frac{3}{4},\frac{r}{2}+\epsilon)}$ is a non-empty simplicial complex.  $\nDRips(Y)_{(k,s)}$ is empty for $k>\frac{1}{3}$ and $s\leq r$, so in particular, since 
\[\frac{3}{4}-\delta%>\frac{3}{4}-\frac{3}{8}=\frac{3}{8}
>\frac{1}{3}\quad\textup{ and }\quad\left(\frac{r}{2}+\epsilon\right)+\delta%<\frac{cr}{2}+r\left(1-\frac{c}{2}\right)-\delta+\delta
\leq r,\] we have that
\[\nDRips(Y)_{(\frac{3}{4}-\delta,c(\frac{r}{2}+\epsilon)+\delta)}=\nDRips(Y)_{\gamma^{\delta,c}(\frac{3}{4},\frac{r}{2}+\epsilon)}\] is empty.

If $\nDRips(Y)$ and $\nDRips(Z)$ were $\gamma^{\delta,c}$-homotopy interleaved, then the persistence modules $H_0(\nDRips(Y))$ and $H_0(\nDRips(Z))$, would be $\gamma^{\delta,c}$-interleaved. 
 But then the internal map \[j\colon H_0(\nDRips(Z)_{(\frac{3}{4},\frac{r}{2}+\epsilon)})\to H_0(\nDRips(Z)_{\gamma^{\delta,c}\circ \gamma^{\delta,c}(\frac{3}{4},\frac{r}{2}+\epsilon)}),\] which is non-zero, would factor through the trivial vector space \[H_0(\nDRips(Y)_{\gamma^{\delta,c}(\frac{3}{4},\frac{r}{2}+\epsilon)}),\] a contradiction.  
(Note that by the way we chose $\delta$, $\gamma^{\delta,c}\circ \gamma^{\delta,c}(\frac{3}{4},\frac{r}{2}+\epsilon)$ has positive $x$-coordinate, so the map 
$j$ is well defined.) This proves the proposition for $c\in [1,2)$.

Having completed our warmup, we now turn to our main argument, which will establish the result for all $c\in [1,3)$.  Fix $c\in [1,3)$ and choose \[r>\frac{2}{301(3-c)}\geq \frac{1}{301}.\]  Let $S'$ be the following equispaced subset of a square in $\R^2$, regarded as an ordered set with the order as given:
\begin{align*}
S'=\{&(0,0),(1,0),(2,0),\ldots, (25,0),(25,1),(25,2),\ldots, (25,25),\\
       &(24,25),(23,25),\ldots,(0,25),(0,24),(0,23),\ldots, (0,1)\}.
\end{align*}
 Let $S=\{r\vec s\mid \vec s\in S'\}$. Note that $|S|=100$.  The order on $S'$ induces an order on $S$; write the $i^{\mathrm{th}}$ point of $S$ as $s_i$.  Let
 \[Y=\{r\mathbf{e}_i\mid 1\leq i\leq 300\}\subset \R^{300}.\] 
For $i\in \{2,\ldots,99\}$, define \[Y_{i}\subset \R^{30002}= \R^2\times \R^{300(i-1)}\times \R^{300}\times \R^{300(100-i)}\]
by \[Y_i\colon =\{(s_i,0,y,0)\mid y\in Y\}.\]
Similarly, define
\begin{align*}
Y_1\subset \R^{30002}&= \R^2\times \R^{300}\times \R^{29700}\text{ and }\\
Y_{100} \subset \R^{30002}&= \R^2\times \R^{29700}\times \R^{300}        
\end{align*}
by 
\begin{align*}
Y_1=\{(s_i,y,0)\mid y\in Y\},\\Y_{100}=\{(s_i,0,y)\mid y\in Y\}.
\end{align*}
Letting $\vec 0_{30000}\in \R^{30000}$ denote the zero vector, we define 
\begin{align*}
W&=\bigcup_{i=1}^{100} Y_i,\\
X&=W\cup (S\times \{\vec 0_{30000}\})).
\end{align*}
Note that 
\begin{align*}
|W|&=300*100=30000,\\
|X|&=|W|+100=30100.
\end{align*}
Let $d_1$ denote the $\ell_1$-metric on $\R^{30002}$.  We regard $W$ and $X$ as metric spaces via restriction of $d_1$.  
 
Write $\hat S= S\times \{\vec 0_{30000}\}$, and for each $s_i\in S$ write $\hat s_i=(s_i,\vec 0_{30000})$.  We record the distances between all pairs of points in $X$:
\begin{lemma}\label{Lem:Dist_Obvs}\mbox{}
Letting $y_i$ and $y_j$ be points of $Y_i$ and $Y_j$, respectively, we have 
\begin{enumerate}[(i)]
%\item $d_1(\hat s_i,y_i)=r$,
\item $d_1(\hat s_i,\hat s_j)=rm$, where $m\in \{0,\ldots,50\}$ and \[m\equiv \pm (j-i) \mod 100,\]
\item $d_1(y_i,y_j)=d_1(\hat s_i,\hat s_j)+2r$,
\item $d_1(\hat s_i,y_j)=d_1(\hat s_i,\hat s_j)+r$.
\end{enumerate}
\end{lemma}

To establish \cref{Prop:More_Explicit}\,(ii), it suffices to prove the following.

\begin{lemma}\label{Lem:Main_Tightness_Step}
$\nDRips(W)$ and $\nDRips(X)$ are not $\gamma^{\delta,c}$-interleaved for some $\delta>\GPr(\mu_W,\mu_X)$.  
\end{lemma}

\begin{proof}[Proof]
By \cref{Rem:Outliers}, we
have \[\Pro(\mu_W,\mu_X)\leq\Pro(\nu_W,\nu_X)\leq
\frac{100}{30100}=\frac{1}{301}.\]  

Choose $\delta$ in the open interval
$\left(\frac{1}{301},\frac{2}{301}\right)$, 
and choose \[\epsilon\in \left(0,\min\left(\frac{r}{2},\frac{r(3-c)-2\delta}{2c}\right)\right].\]  

Using \cref{Lem:Dist_Obvs}, and noting that $\frac{1}{100}=\frac{301}{30100}$, it is easily checked that $\nDRips(X)_{(\frac{1}{100},\frac{r}{2}+\epsilon)}$ has vertex set $\hat S$, and edge set \[\{[\hat s_i,\hat s_j]\mid (j-i)\equiv 1\mod 100\}.\]  Thus, $\nDRips(X)_{(\frac{1}{100},\frac{r}{2}+\epsilon)}$ is homeomorphic to a circle.  

We note that for $s\leq \frac{3r}{2}$ and any $k>0$, $\nDRips(W)_{(k,s)}$ is either empty or %a disjoint union of simplices; the three possibilities are that $\nDRips(W)_{(k,s)}$ is empty, $\nDRips(W)_{(k,s)}$ is a disjoint union of 300 99-dimensional simplices, or $\nDRips(W)_{(k,s)}$ is 
a disjoint union of 100 299-dimensional simplices.
So in particular, since 
\[c\left(\frac{r}{2}+\epsilon\right)+\delta<\frac{cr}{2}+\frac{r}{2}(3-c)-\delta+\delta\leq \frac{3r}{2}.\]
we have that
\[\nDRips(W)_{\gamma^{\delta,c}(\frac{1}{100},\frac{r}{2}+\epsilon)}=\nDRips(W)_{(\frac{1}{100}-\delta,c(\frac{r}{2}+\epsilon)+\delta)}\] is either empty or a disjoint union of simplices.  
Let us write 
\begin{align*}
\vec a&:=\left(\frac{1}{100},\frac{r}{2}+\epsilon\right)\\
\vec b&:=\gamma^{\delta,c}\circ
  \gamma^{\delta,c}\left(\frac{1}{100},\frac{r}{2}+\epsilon\right)
\end{align*}
%\[Q:=\nDRips(X)_{\gamma^{\delta,c}\circ
 % \gamma^{\delta,c}(\frac{1}{100},\frac{r}{2}+\epsilon)}.\]  
  If we can
show that the inclusion \[j\colon \nDRips(X)_{\vec a}\hookrightarrow \nDRips(X)_{\vec b}\] induces a non-trivial map on $H_1$, then we may conclude that $\nDRips(W)$ and $\nDRips(X)$ are not $\gamma^{\delta,c}$-homotopy interleaved, by essentially the same argument as we used in the special case $c\in [1,2)$, thereby completing the proof.  To show that $j$ has the desired property, we will identify simplicial complexes $T'$ and $T$ such that
\[\nDRips(X)_{\vec a}\subset T'\subset \nDRips(X)_{\vec b}\subset T\]
and the inclusions \[\nDRips(X)_{\vec a}\hookrightarrow T'\hookrightarrow T\] are both homotopy equivalences.  Thus, the composition of these maps is also a homotopy equivalence, and since $j$ is a factor of this composition, it must induce a non-trivial map on $1^{\mathrm{st}}$ homology.

Writing $\vec b= (b_1,b_2)$, we let $T'=\Rips(\hat S)_{b_2}$.  Note that since the vertex set of $\nDRips(X)_{\vec a}$ is $\hat S$, we have \[\nDRips(X)_{\vec a}\subset T'\subset \nDRips(X)_{\vec b}.\]  
To define $T$, let $f\colon X\to \hat S$ be the surjection such that $f(x)=\hat s_i$ for each $x\in Y_i\cup \{\hat s_i\}$.  Let \[T=\{\sigma\subset X\mid f(\sigma)\in T'\}.\]  

 To see that $\nDRips(X)_{\vec b}\subset T$, note that $f$ is 
 distance non-increasing, so $f$ is a simplicial map $\Rips(X)_t\to
 \Rips(\hat S)_t$ for any $t> 0$.  In particular, \[f\colon \Rips(X)_{b_2}\to
 \Rips(\hat S)_{b_2}=T'\] is a simplicial map, which implies that 
 $f(\sigma)\in T'$ for all $\sigma\in \Rips(X)_{b_2}$.  We thus have \[\nDRips(X)_{\vec b}\subset \Rips(X)_{b_2}\subset T.\]  
 
Now $f$ is a simplicial retraction from $T$ to $T'$, and for
$\sigma\in T'$, $f^{-1}(\sigma)$ is a simplex in $T$.  Therefore, by
Quillen's theorem A for simplicial complexes \cite{quillen1973higher},
$f\colon T\to T'$ is a homotopy equivalence. Since $f\circ i=\Id_{T'}$
is also homotopy equivalence, it then follows from the 2-out-of-3
property that the inclusion $i\colon T'\hookrightarrow T$ is a homotopy equivalence.

It remains for us to check that the inclusion \[\nDRips(X)_{\vec a}\hookrightarrow T'\] is a homotopy equivalence.  As we now explain, this is a special case of a result in \cite{adamaszek2017vietoris}.  Adapting the notation of \cite{adamaszek2017vietoris}, let 
$C_{100}^k$ denote the graph with vertices $1,\ldots,100$, and an edge connecting $i$ and $j$ if and only if \[(j-i) \equiv \pm l \mod m\] for some $l\in \{1,\ldots, k\}$.  Note that $\nDRips(X)_{\vec a}$ is isomorphic to the graph $C_{100}^1$, and by symmetry, the graph underlying $T'$ is isomorphic to $C_{100}^{k}$ for some $k\in \{1,\ldots 50\}$.  We now show that in fact, $k\leq 33$: Note that 
\[\vec b=
\left(\frac{1}{100}-2\delta,\ c^2\left(\frac{r}{2}+\epsilon\right)+(c+1)\delta\right).\]  Since $c<3$, $\epsilon<\frac{r}{2}$, and \[\delta<\frac{2}{301}<2r,\] we have that 
\[c^2\left(\frac{r}{2}+\epsilon\right)+(c+1)\delta<9\left(\frac{r}{2}+\frac{r}{2}\right)+(3+1)2r=17r.\]  
Thus, $[\hat s_i,\hat s_j]\in T'$ only if $j-i \equiv \pm l \mod 100$ for some $l<34$.  Therefore, $k\leq 33$ as claimed.  

According to \cite[Proposition 3.14, Theorem 4.3, and Theorem 4.9]{adamaszek2017vietoris}, whenever $1\leq k< \frac{100}{3}$, the inclusion $C_{100}^1\hookrightarrow C_{100}^k$ is a homotopy equivalence.  Thus, the inclusion $\nDRips(X)_{\vec a}\hookrightarrow T'$ is a homotopy equivalence, as desired.
\end{proof}

\bibliographystyle{abbrv}
\bibliography{PS_Refs}

\end{document}